\documentclass[12pt, psamsfonts, reqno]{amsart}

\usepackage{amssymb,amsfonts}
\usepackage[all,arc]{xy}
\usepackage{enumerate}
\usepackage{mathrsfs}
\usepackage{graphicx} 
\usepackage[hang,flushmargin]{footmisc}
\usepackage{caption}
\usepackage{subcaption}
\usepackage{hyperref}
\hypersetup{colorlinks,allcolors=blue}
\usepackage{tikz}
\usepackage{tikz-cd}
\usepackage{stackrel}
\usepackage{thmtools}
\usepackage{amsthm}
\usepackage{xcolor,colortbl}
\usepackage{adjustbox}
\usepackage{array,makecell}
\usepackage{spectralsequences}
\usepackage{xcolor}
\usepackage{alphabeta}
\usepackage{bm}
\usepackage{geometry}
\usepackage{amsmath}
\usepackage{nccmath}

\usetikzlibrary{trees,positioning,fit,arrows,decorations.pathreplacing}
\usetikzlibrary{positioning}
\usetikzlibrary{shapes}
\usetikzlibrary{shapes.geometric}
\usetikzlibrary{decorations.fractals}
\usetikzlibrary{decorations.text}
\usetikzlibrary{decorations.shapes}
 \usetikzlibrary{shapes, shapes.geometric, shapes.symbols, shapes.arrows, shapes.multipart, shapes.callouts, shapes.misc}
 \usetikzlibrary{patterns}
 
 \usetikzlibrary{decorations.markings}
\usetikzlibrary{shapes,arrows}

\newcolumntype{M}[1]{>{\centering\arraybackslash}m{#1}}
\newcolumntype{N}{@{}m{0pt}@{}}

\newtheorem{thm}{Theorem}[section]
\newtheorem{cor}[thm]{Corollary}
\newtheorem{prop}[thm]{Proposition}
\newtheorem{lem}[thm]{Lemma}

\theoremstyle{definition}
\newtheorem{defn}[thm]{Definition}

\newtheorem{rem}[thm]{Remark}

\theoremstyle{remark}

\numberwithin{equation}{section}

\usepackage{float}

\renewcommand{\epsilon}{\varepsilon}

\renewcommand{\phi}{\varphi}

\usepackage{enumitem}

\title[Antonios Zitridis]{The Master Equation in a Bounded Domain under Invariance Conditions for the State Space}

\author{Antonios Zitridis}
\address{Department of Mathematics, University of Chicago, Chicago, Illinois, 60637, USA.}
\email[]{zitridisa@uchicago.edu}

\date{\today}

\begin{document}

\begin{abstract}
In this paper, we study the well-posedness (existence and uniqueness) of the Master Equation of Mean Field Games under invariance-type conditions, otherwise known as viability conditions for the controlled dynamics. The interior regularity of the solutions of the associated Mean Field Game system and its linearized version, which plays a crucial role in the proof of the existence, is obtained by the global regularity of the corresponding solutions in the Neumann boundary conditions case. Finally, we prove that the solution of the related Nash system converges to the solution of the Master Equation.
\end{abstract}

\maketitle

\tableofcontents

\section{Introduction}

\noindent

\noindent
The Mean Field Game (MFG for short) theory was introduced by J.-M. Lasry and P.-L Lions [10, 11, 12], and in a particular case by Huang, Caines and Malhame [13], in order to describe Nash equilibria in differential games with infinitely many (small and undistinguishable) players, called ``infinitesimal'' players. Its name comes from the analogy with the mean field models in mathematical physics, which analyzes the behaviour of many identical particles.

\vspace{2mm}
\noindent
The macroscopic description used in MFG theory leads to the study of coupled systems of PDEs, consisting of a backward Hamilton-Jacobi-Bellman equation satisfied by the single agent's value function $u$ and a forward Kolmogorov-Fokker-Planck equation satisfied by the distribution law $m$ of the population. The simplest form of this system in a bounded domain $\Omega\subset \mathbb{R}^d$ reads
$$\bm{(MFG)}: \begin{cases}
-\partial_t u-\text{tr}(a(x)D^2u)+H(x,Du)=F(x,m), \text{ in } (0,T)\times\Omega\;,\\
\partial_t m-\sum_{i,j}\partial_{i,j}^2(a_{ij}(x)m)-\text{div}(mH_p(x,Du))=0,\text{ in }(0,T)\times\Omega\;,\\
m(0)=m_0,\quad u(T,x)=G(x,m(T)),
\end{cases}$$
with appropriate boundary conditions to achieve well-posedness, where $m_0$ is the initial state of the population, $H(x,p)$ is the Hamiltonian of the system, $F$ and $G$ are the running cost and the final cost and $a=\sigma \sigma^T$ is uniformly elliptic, where $\sigma$ is the diffusion matrix of the stochastic dynamics of the ``infinitesimal'' player :
$$\bm{(SDE)}: \; dX_t=-H_p(x,Du)dt+\sqrt{2}\sigma (X_t) dB_t, \quad X_0=x\in \Omega.$$

\vspace{2mm}

\noindent
Although the MFG system has been widely studied since its introduction, it has become increasingly clear that it is not sufficient to take into account the entire complexity of games with infinitely many players. This led Lasry and Lions [8] to introduce the so-called “Master Equation” (ME for short), which directly describes the limit of the Nash equilibrium for the differential game with $N$ players as $N\rightarrow +\infty$. 
In particular, if $(u,m)$ is the unique (under certain assumptions) solution to the MFG system with initial condition $m(t_0)=m_0$, then we define the function 
\begin{equation} \label{ME}
U: [0,T]\times \Omega\times \mathcal{P}(\Omega)\rightarrow \mathbb{R}\text{ by } U(t_0,x,m_0):=u(t_0,x),
\end{equation}
where $\mathcal{P}(\Omega)$ is the set of Borel probability measures on $\Omega$. A formal computation yields the \textbf{ME}, which is the equation satisfied by $U$:

$$\bm{(ME)}:\begin{cases}
\partial _tU(t,x,m)=-\text{tr}(a(x)D^2_xU(t,x,m))+H(x,D_xU(t,x,m))\\\\
-\int_{\Omega}\text{tr}(a(y)D_yD_mU(t,x,m,y))dm(y)\\
\\
+\int_{\Omega}D_mU(t,x,m,y)\cdot H_p(y,D_xU(t,y,m))dm(y)-F(x,m)\;,\\
\quad\quad\quad\quad\quad\quad\quad\quad\quad\quad\quad\quad\quad\quad\quad\quad\quad\quad\text{ in }[t_0,T]\times \Omega\times \mathcal{P}(\Omega),\\
 \\
U(T,x,m)=G(x,m)\quad\text{ in }\Omega\times\mathcal{P}(\Omega)\;.
\end{cases}
$$
Of course, one needs to assume the appropriate boundary conditions to achieve well-posedness.
\vspace{3mm}

\noindent
So far, most of the literature related to the \textbf{ME} considers the case where the state variable $x$ belongs to the flat torus (i.e $\Omega=\mathbb{T}^d$) ([5]) or in $\mathbb{R}^d$. Recently, Ricciardi [2] and Di Persio, Garbelli, Ricciardi [9] studied the problem in an arbitrary bounded domain $\Omega$ with Neumann conditions and Dirichlet conditions, respectively. A common characteristic of these two cases is that the stochastic trajectories of the ``infinitesimal'' players, described by \textbf{SDE}, do not exit the domain $\Omega$ but may reach the boundary\footnote{The trajectories reflect, if we have Neumann conditions, and stay on the boundary, if we have Dirichlet conditions.}. 

\vspace{3mm}

\noindent
Another significant case when the dynamical state remains inside the domain, but, this time, does not reach the boundary, is if one considers a \textit{state constraint} condition. An example of such a condition was studied in [17] by Sardarli, where the value function $u$ blows up near the boundary. In this paper, we will assume the invariance condition used in [1] by Porretta and Ricciardi:
\begin{equation} 
\text{tr}(a(x)D^2d(x))-H_p(x,p)\cdot Dd(x)\geq \frac{a(x)Dd(x)\cdot Dd(x)}{d(x)}-Cd(x)
\end{equation}
for some constant $C>0$, for all $p\in\mathbb{R}^d$ and for all $x$ in a neighbourhood of the boundary $\partial \Omega$. It is proved in Proposition 2.5 of [1] that this is a condition in which the drift-diffusion terms are such that, almost surely, the stochastic trajectory of the ``infinitesimal'' player (\textbf{SDE}) does not hit the boundary and, hence, stays inside the domain\footnote{In the control community, sometimes this property is referred to as the \textit{viability of the state space.}}. Furthermore, in Theorem 5.1 ([1]), it is proved that \textbf{(MFG)} has a unique solution if $(1.2)$ holds near the boundary.
\vspace{3mm}

\noindent 
The first aim of this paper is to analyze the well-posedness of the \textbf{ME} when $(1.2)$ holds near the boundary. In particular, we will show that, under specific assumptions, the solution of \textbf{(MFG)} obtained in [1] is regular (Proposition \ref{MFGL1}) and that $U$ is the unique classical solution of the \textbf{ME}.
\vspace{2mm}

\noindent
Returning to the $N$-player differential game, where we assume  that every player controls his/her own state and interacts only through their cost function, it is known that at a Nash equilibrium the value functions of the players, $v_i^N$, $1\leq i\leq N$ satisfy the \textit{Nash system}
\begin{equation} \label{NS}
    \begin{cases}
-\partial_t v_i^{N,\epsilon}-\sum_j \text{tr}(a(x_j)D^2_{x_jx_j}v_i^{N,\epsilon}(t,\bm{x}))+H(x_i,D_{x_i}v_i^{N,\epsilon}(t,\bm{x}))\\
\quad\quad\quad\quad+\sum_{j\neq i}H_p(x_j,D_{x_j}v_j^N(t,\bm{x}))\cdot D_{x_j}v_i^{N,\epsilon}(t,\bm{x})=F(x_i,m_{\bm{x}}^{N,i}),\quad\text{in }[0,T]\times \Omega^N,\\
v_i^{N,\epsilon}(T,\bm{x})=G(x_i,m_{\bm{x}}^{N,i})\quad\text{ in }\Omega^N,
\end{cases}
\end{equation}
where
$$m_{\bm{x}}^{N,i}=\frac{1}{N-1}\sum_{j\neq i}\delta_{x_j}$$
is the empirical distribution of the players $j\neq i$ and where the functions $v_i^N,\;\;i=1,2,...,N$ satisfy an appropriate boundary condition (Dirichlet or Neumann) or there is a state constraint condition. In this game, the set of ``optimal trajectories'' solves a system of $N$ coupled stochastic differential equations:
\begin{equation} \label{system1}
  \begin{cases} 
    dY^i_t=-H_p(Y^i_t,Dv_i^N(t,\bm{Y}_t))dt+\sqrt{2}\sigma(Y^i_t)dB^i_t,\;\;t\in[0,T],\\
    Y_0^i=Z_i,
    \end{cases}\text{, }1\leq i\leq N\;,
\end{equation}
where $\bm{Z}=(Z_i)_i$ are i.i.d of fixed law. 
\vspace{2mm}

\noindent
The structure of the $N$-player game when $N>>1$ is very complex and in that case it is natural to study the asymptotic behaviour of (\ref{NS}) as $N\rightarrow +\infty$. This question has been addressed on the flat torus (i.e $\Omega=\mathbb{T}^d$) by P. Cardaliaguet, F. Delarue, J.-M Lasry, P.-L Lions in [5], and, more recently, on arbitrary bounded domains under Neumann conditions (by Ricciardi in [16]) and Dirichlet conditions (by   Di Persio, Garbelli and Ricciardi in [9]). 
\vspace{2mm}

\noindent
The second aim of this paper is to study the asymptotic behaviour of the differential game on arbitrary bounded domain $\Omega$; assuming that $(1.2)$ holds near the boundary. The main result, as in [5], [16] and [9], says that the $(v_i^N)_{i\in \{1,2,...,N\}}$ ``converge'' to the solution of the \textbf{ME} as $N\rightarrow +\infty$. This result, conjectured by Lasry and Lions [12], connects the Nash system, in which the players observe each other, with the limit system, in which the players observe just the distribution of the population. A way to express the convergence is given in Theorem 7.4. In particular, if
\begin{equation}
u_i^N(t,\bm{x}):=U(t,x_i,m_{\bm{x}}^{N,i})\;,
\end{equation}
then $|u_i^N-v_i^N|\rightarrow 0$ in some suitable norm. Finally, one can describe the convergence in terms of the trajectories. Namely, the solution of (\ref{system1}) is ``close'' to the solution of
\begin{equation}
\begin{cases} 
    dX^i_t=-H_p(X^i_t,Du_i^N(t,\bm{X}_t))dt+\sqrt{2}\sigma(X^i_t)dB^i_t,\;\;t\in[0,T],\\
    X_0^i=Z_i,
    \end{cases}\text{, }1\leq i\leq N    .
\end{equation}
The proofs of both convergence results rely heavily on methods from [5], [9] and [16].

\vspace{7mm}

\noindent
\textbf{Main results.} Our main results summarize as follows

\begin{thm}\label{MEsol}
Suppose Hypotheses 2.3 are satisfied. Then, the function $U$ from (\ref{ME}) is a classical solution of the \textbf{ME}.
\end{thm}
\vspace{3mm}

\begin{thm}
Suppose Hypotheses 2.3 are satisfied. Then, $U$ is the unique classical solution of the \textbf{ME}.
\end{thm}
\vspace{3mm}

\begin{thm}
Assume Hypotheses 2.3. Let 
$$m_{\bm{x}}^{N}=\frac{1}{N}\sum_{i=1}^N\delta_{x_i}$$
and
$$w_i^N(t,x_i,m):=\int_{\Omega^{N-1}}v_i^N(t,\bm{x})\prod_{j\neq i}m_0(dx_j).$$
Then,
$$\sup_{t\in [0,T],i\in\{1,2,...,N\}}|v_i^N(t,\bm{x})-U(t,x_i,m_{\bm{x}}^N)|\leq \frac{C}{N}$$
and
$$||w_i^N(t,\cdot, m)-U(t,\cdot,m)||_{L^1(m)}\leq C\omega_N,$$
where
$$\omega_N=
\begin{cases}
CN^{-\frac{1}{d}} &,\text{ if }d\geq 3,\\
CN^{-\frac{1}{2}}\log N &,\text{ if }d=2,\\
CN^{-\frac{1}{2}} &,\text{ if }d=1.
\end{cases}$$
\end{thm}
\vspace{3mm}

\begin{thm}
Assume Hypotheses 2.3 and, furthermore, that $\sigma$ is Lipschitz. Then, for any $1\leq i\leq N$, we have
$$\sup_{t\in [0,N]}\mathbb{E}[|X_t^i-Y_t^i|^2]\leq\frac{C}{N^2}\;.$$
\end{thm}

\vspace{4mm}

\noindent
\textbf{Summary of the main ideas.} The paper follows the main ideas of [2], [5] and [9], but appropriate adjustments need to be made because of (1.2). To deal with the nature of (1.2), we use the main ideas presented in [1].
\vspace{2mm}

\noindent
For $U$, defined in (\ref{ME}), to be a classical solution of the \textbf{ME}, some regularity and some bounds need to be proved. Some initial regularity can be obtained by studying \textbf{(MFG)} under (1.2). Indeed, if $m_0\in L^1(\Omega)$, as proved in [1], the solution to \textbf{(MFG)} can be obtained as the limit of the classical solutions to
\begin{equation*} 
\begin{cases}
-\partial_t u^{\epsilon}-\text{tr}(a(x)D^2u^{\epsilon})+H(x,Du^{\epsilon})=F(x,m^{\epsilon})\quad ,\text{ in } (0,T)\times \Omega_{\epsilon},\\
\partial_t m^{\epsilon}-\text{div}(a(x)Dm^{\epsilon})-\text{div}(m^{\epsilon}(H_p(x,Du^{\epsilon})+\tilde{b}(x)))=0\quad ,\text{ in } (0,T)\times\Omega_{\epsilon},\\
m^{\epsilon}(0)=m_0 ,\quad u^{\epsilon}(T,x)=G(x,m^{\epsilon}(T))-\mathcal{N}_{\epsilon}(a(\cdot)D_xG(\cdot,m^{\epsilon}(T))),\\
a(x)Du^{\epsilon}\cdot \nu(x)|_{\partial \Omega_{\epsilon}}=0, \quad \left(a(x)Dm^{\epsilon}+m^{\epsilon}(H_p(x,Du^{\epsilon})+\tilde{b}(x))\right)\cdot \nu(x)|_{\partial \Omega_{\epsilon}}=0,
\end{cases}
\end{equation*}
where $\mathcal{N_{\epsilon}}$ is an operator that ensures the Neumann compatibility conditions and where $\epsilon\rightarrow 0^+$, given by results from [2]. Thus, under specific assumptions, $U$ inherits the interior regularity of $u^{\epsilon}$. The result can be extended for $m_0\in\mathcal{P}(\Omega)$ by a density argument.
\vspace{2mm}

\noindent
However, the main issue is to prove the $C^1$ character of $U$ with respect to the measure variable. The most important tool of the proof is the analysis of the linearized system.
\begin{equation*} \label{L}
\bm{(L)} :\begin{cases}
-\partial_t v-\text{tr}(a(x)D^2v)+H_p(x,Du)\cdot Dv=\frac{\delta F}{\delta m}(x,m(t))(\mu (t)),\text{ in } (t_0,T)\times\Omega,\\
\partial_t\mu-\text{div}(a(x)D\mu)-\text{div}(\mu (H_p(x,Du)+\tilde{b}(x)))-\text{div}(m H_{pp}(x,Du)Dv)=0,\\
\hspace{11.5cm}\text{ in }(t_0,T)\times\Omega,\\
v(T,x)=\frac{\delta G}{\delta m}(x,m(T))(\mu (T))\text{ in }\Omega,\quad \mu (t_0)=\mu _0\in C^{-(1+\alpha)}(\Omega).
\end{cases}
\end{equation*}
Using the same idea as in \textbf{(MFG)}, when $m_0\in L^1(\Omega)$, we can approximate a solution of the linearized system in $\Omega$ under (1.2) with solutions of the linearized system in $\Omega_{\epsilon}$ with Neumann conditions (existence is established in [2]). The convergence will come from estimate (5.20) in [2]. Then, we can generalize the result for $m_0\in\mathcal{P}(\Omega)$ with a density argument. The regularity will follow from the estimates for the linearized system.
\vspace{2mm}

\noindent
Using the regularity of $U$, we can show that $U$ satisfies the \textbf{ME} and, by an argument similar to the one in [5], that it is the unique function with that property. 
\vspace{2mm}

\noindent
Finally, we establish the well-posedness of (\ref{NS}) under (1.2) and we use similar ideas as in [5], [9] and [16] to prove the convergence results.

\vspace{4mm}

\noindent
\textbf{Organization of the paper.} The paper is organized as follows:

\vspace{1mm}

\noindent
Section 2 contains our notation and assumptions we will need to prove the results. Section 3 contains useful preliminary results that will appear throughout the paper. In section 4, using our assumptions, we strengthen the results from [1] and obtain some regularity for $U$. In section 5, we solve the linearized system provided that $(1.2)$ holds near the boundary, which provides us with further regularity for $U$. In section 6, we prove that the \textbf{ME} has a unique classical solution. In section 7, we establish the well-posedness of (\ref{NS}) under (1.2) near the boundary and we prove the convergence results.

\vspace{10mm}

\section{Notation and Hypotheses}

\subsection{The domain}
We assume throughout the paper that $\Omega$ is a bounded open subset of $\mathbb{R}^d$, with boundary of class $C^{3+\alpha}$ ([7, 14]), that is, $\partial \Omega$ is locally a graph of a $C^{3+\alpha}$-function. We recall that the oriented distance from $\partial \Omega$, denoted by $d_{\Omega}$, is the function defined by 
$$d_{\Omega}(x)=\begin{cases}
d(x,\partial \Omega),\text{        if }x\in\Omega,\\
-d(x,\partial \Omega),\text{ if }x\notin\Omega.
\end{cases}$$
We write $d$ instead of $d_{\Omega}$ when there is no possible confusion.
We will, also, be working in subdomains
$$\Omega_{\epsilon}:= \{ x\in\Omega : d(x)>\epsilon\},\text{ for }\epsilon>0$$
For convenience, when $\epsilon=0$, we write $\Omega_{\epsilon}=\Omega$.
Due to the regularity of the boundary of $\Omega$, there exists $\epsilon_0>0$ such that for every $\epsilon\in (0,\epsilon_0)$ we have $d\in C^{3+\alpha}(\Omega\setminus \Omega_{\epsilon})$ and for all $x\in \Omega\setminus \Omega_{\epsilon} $, there exists a unique point $\tilde{x}\in \partial\Omega$ such that $d(x)=|x-\tilde{x}|$, where $|\cdot |$ is the Euclidean norm. Moreover, on $\Omega\setminus \Omega_{\epsilon}$,
$$Dd(x)=Dd(\tilde{x})=-\nu(\tilde{x}),$$
where $\nu(\tilde{x})$ is the exterior unit normal vector of $\partial \Omega$ at $\tilde{x}$. 
\vspace{2mm}

\noindent
Finally, a classical regularizing argument allows us to consider a non-negative $C^{3+\alpha}(\overline{\Omega})$-function, called again $d_{\Omega}$, which coincides with the oriented distance near the boundary.

\subsection{Notation}
Throughout the paper, $T>0$ will be a fixed time and $Q_T:=[0,T]\times \Omega$. Also, $Q_T^{\epsilon}$ stands for $[0,T]\times\Omega_{\epsilon}$.

\subsubsection{Spaces and Norms}
\noindent
For $n\geq 0$ and $\alpha\in (0,1)$ we denote with $C^{n+\alpha}(\Omega)$ or simply $C^{n+\alpha}$, when there is no confusion regarding the domain, the space of functions $\phi\in C^n(\Omega)$ such that for each $l\in\mathbb{N}^r$, with $1\leq r\leq n$, the derivative $D^l\phi$ is Hölder continuous with constant $\alpha$. The norm on this space is given by
$$||\phi||_{n+\alpha}:=\sum_{|l|\leq n}||D^l\phi||_{\infty}+\sum_{|l|=n}\sup_{x\neq y}\frac{|D^l\phi (x)-D^l\phi (y)|}{|x-y|^{\alpha}}.$$

\noindent
We also denote by $C^{n+\alpha}_c(\Omega)$ the closure of the subspace of compactly supported functions in $C^{n+\alpha}(\Omega)$, by $C^{-(n+\alpha)}$ and $C^{-(n+\alpha)}_c$ the corresponding duals and by $\langle \cdot,\cdot\rangle_{1+\alpha}$ the corresponding duality bracket. Finally, given a function $a(\cdot)$, which we will define later (Hypothesis (1)), we consider the space $C^{n+\alpha,N}(\Omega)$ i.e the space of functions $\phi\in C^{n+\alpha}(\Omega)$ such that $a(x)D_x\phi\cdot\nu(x)|_{\partial\Omega}=0$.

\vspace{3mm}

\noindent
We also define several parabolic spaces:
\vspace{1mm} 

\noindent
We write $\phi\in C^{\frac{n+\alpha}{2},n+\alpha}([0,T]\times\Omega)$, if $\phi$ is continuous in both variables, together with all the derivatives $D^r_tD^s_x\phi$, where $2r+s\leq n$. Moreover, we introduce the norm
\begin{multline*}
||\phi||_{\frac{n+\alpha}{2},n+\alpha}:=\sum_{2r+s\leq n}||D^r_tD^s_x\phi||_{\infty}+\sum_{2r+s=n}\sup_{t\in [0,T]}||D^r_tD^s_x\phi(t,\cdot )||_{\alpha}\\
+\sum_{0<n+\alpha-2r-s<2}\sup_x||D^r_tD^s_x\phi(\cdot, x)||_{\frac{n+\alpha-2r-s}{2}}.
\end{multline*}

\vspace{2mm}

\noindent
We define the space of space-time continuous functions which satisfy a Hölder condition in $x$ to be $C^{0,\alpha}([0,T]\times\Omega)$ with the norm
$$||\phi||_{0,\alpha}=\sup_{t\in [0,T]}||\phi(t,\cdot )||_{\alpha}.$$
An analogous definition can be given for the space $C^{\alpha,0}$. 
\vspace{3mm}

\noindent
The space $C^{1,2+\alpha}$ consists of functions which are differentiable in time and twice differentiable in space, with all derivatives in $C^{0,\alpha}(\overline{Q_T})$. The norm for this space is 
$$||\phi||_{1,2+\alpha}:=||\phi||_{\infty}+||\partial_t\phi||_{0,\alpha}+||D_x\phi||_{\infty}+||D^2_x\phi||_{0,\alpha}.$$

\vspace{3mm}

\noindent
Finally, $W^{-k,p}$ is the dual of the Sobolev space $W^{k,p}$.

\subsubsection{Generalized Wasserstein Space and Derivative}

Let $\mathcal{P}^{sub}(\Omega)$ be the space of all Borel subprobability measures, that is the space of all Borel non-negative measures in $\Omega$ with total mass at most $1$, and $\mathcal{P}(\Omega)$ be the space of probability measures.

\begin{defn}
Let $m_1,m_2\in \mathcal{P}^{sub}(\Omega)$ be two Borel sub-probability measures in $\Omega$. We call the generalized 1-Wasserstein distance $\bm{d_1}(m_1,m_2)$ between $m_1$ and $m_2$ to be
$$\bm{d_1}(m_1,m_2):=\sup_{Lip(\phi)\leq 1}\int_{\Omega}\phi(x)d(m_1-m_2)(x).$$
\end{defn}

\vspace{3mm}

\begin{rem}
In $\mathcal{P}^{sub}(\Omega_{\epsilon})$ we will denote the generalized Wasserstein distance by $\bm{d_1}^{\epsilon}$.
\end{rem}

\vspace{3mm}

\begin{defn}
\textbf{(i)} Let $V:\mathcal{P}^{sub}(\Omega)\rightarrow \mathbb{R}$. We say that $V$ is of class $C^1$ if there exists a continuous map $K:\mathcal{P}^{sub}(\Omega)\times \Omega\rightarrow \mathbb{R}$ such that, for all $m_1,m_2\in \mathcal{P}^{sub}(\Omega)$ we have 
$$\lim_{s\rightarrow 0}\frac{V(m_1+s(m_2-m_1))-V(m_1)}{s}=\int_{\Omega}K(m_1,x)(m_2(dx)-m_1(dx)).$$
Moreover, we write $K(m,x)=\frac{\delta V}{\delta m}(m,x)$ for all $(m,x)\in\mathcal{P}^{sub}(\Omega)\times \Omega$.
\vspace{1mm}

\noindent
\textbf{(ii)} If $\frac{\delta V}{\delta m}(m,\cdot)\in C^1$, then the intrinsic derivative $D_mV:\mathcal{P}^{sub}(\Omega)\times \Omega \rightarrow \mathbb{R}^d$ is defined as 
$$D_mV(m,x):=D_x\frac{\delta V}{\delta m}(m,x).$$
\end{defn}

\begin{rem}
If instead of $\mathcal{P}^{sub}(\Omega)$, we were working on $\mathcal{P}(\Omega)$, then $K$ is defined the same way (e.g [5]) but up to a constant. We usually choose the constant such that $\int_{\Omega} K(m,x)dx=0$.
\end{rem}

\vspace{3mm}

\subsection{Hypotheses}

We assume the following
\begin{enumerate}
    \item (Uniform ellipticity) $a,\sigma: \Omega\rightarrow \mathbb{M}_d(\mathbb{R})$, such that $a=\sigma\sigma^T$,
    $||a(\cdot)||_{1+\alpha}< \infty$ and there exist $\theta>\lambda>0$ such that $\forall p\in\mathbb{R}^d$
    $$\theta|p|^2\geq  a(x)p\cdot p \geq \lambda |p|^2\quad ;$$
    \item $H:\Omega\times \mathbb{R}^d\rightarrow \mathbb{R}$ is smooth, Lipschitz continuous, $H(\cdot ,0)\in L^{\infty}(\Omega)$ and there exists $ C_{H}>0$ such that
    $$C_H^{-1}\frac{I_{d\times d}}{1+|p|}\leq H_{pp}(x,p)\leq C_{H}I_{d\times d}\quad ;$$
    \item $F:\Omega\times \mathcal{P}^{sub}(\Omega)\rightarrow \mathbb{R}$ is smooth and monotone, i.e,
    $$\int_{\Omega}(F(x,m)-F(x,m'))d(m-m')(x)\geq 0.$$
    In addition, for some $0<\alpha<1$ and $C_F>0$, it satisfies
    $$\sup_{m\in \mathcal{P}^{sub}(\Omega)}\left(||F(\cdot,m)||_{\alpha}+\left|\left|\frac{\delta F}{\delta m}(\cdot,m,\cdot)\right|\right|_{\alpha,2+\alpha}\right)+\text{Lip}_{\alpha}\left(\frac{\delta F}{\delta m}\right)\leq C_F,$$
    with
    $$\text{Lip}_{\alpha}\left(\frac{\delta F}{\delta m}\right):=\sup_{m_1\neq m_2}\left( \bm{d_1}(m_1,m_2)^{-1}\left|\left|\frac{\delta F}{\delta m}(\cdot,m_1,\cdot)-\frac{\delta F}{\delta m}(\cdot,m_2,\cdot)\right|\right|_{\alpha,1+\alpha}\right);$$
    \item $G:\Omega\times \mathcal{P}^{sub}(\Omega)\rightarrow \mathbb{R}$ is smooth and monotone, i.e,
    $$\int_{\Omega}(G(x,m)-G(x,m'))d(m-m')(x)\geq 0.$$
    In addition, for some $0<\alpha<1$ and $C_G>0$,
    $$\sup_{m\in \mathcal{P}^{sub}(\Omega)}\left(||G(\cdot,m)||_{2+\alpha}+\left|\left|\frac{\delta G}{\delta m}(\cdot,m,\cdot)\right|\right|_{2+\alpha,2+\alpha}\right)+\text{Lip}_{1+\alpha}\left(\frac{\delta G}{\delta m}\right)\leq C_G,$$
    with
    $$\text{Lip}_{1+\alpha}\left(\frac{\delta G}{\delta m}\right):=\sup_{m_1\neq m_2}\left( \bm{d_1}(m_1,m_2)^{-1}\left|\left|\frac{\delta G}{\delta m}(\cdot,m_1,\cdot)-\frac{\delta G}{\delta m}(\cdot,m_2,\cdot)\right|\right|_{2+\alpha,2+\alpha}\right);$$
    \item The following conditions are assumed for $F,G$ :
$$a(x)D_xG(x,m)\cdot \nu(x)|_{\partial\Omega}=0\text{ , for all }m\in \mathcal{P}^{sub}(\Omega)\text{ and }x\in\partial\Omega.$$
\vspace{2mm}

\noindent
For every $\epsilon>0$, there exists a functions $\frac{\delta F^{\epsilon}}{\delta m},\frac{\delta G^{\epsilon}}{\delta m}: \Omega\times \mathcal{P}^{sub}(\Omega)\times \Omega\rightarrow \mathbb{R}$ such that $\text{ for all }(x,m)\in\Omega\times\mathcal{P}^{sub}(\Omega)$:
\begin{align*}
&\text{spt} \left(\frac{\delta F^{\epsilon}}{\delta m} (x, m,\cdot) \right) \subset \Omega_{\epsilon}\;, \text{spt}\left(\frac{\delta G^{\epsilon}}{\delta m}(\cdot , m,\cdot)\right)\subset\Omega_{\epsilon}\times\Omega_{\epsilon}\;,\\
&\lim_{\epsilon\rightarrow 0^+}\sup_{m\in \mathcal{P}^{sub}(\Omega)}\left|\left| \frac{\delta G^{\epsilon}}{\delta m}(\cdot, m,\cdot)-\frac{\delta G}{\delta m}(\cdot, m,\cdot)\right|\right|_{C^{2+\alpha}(\Omega)\times C^{2+\alpha}(\Omega)}=0\;\;\;\text{     and }\\
&\lim_{\epsilon\rightarrow 0^+}\sup_{m\in \mathcal{P}^{sub}(\Omega)}\left|\left| \frac{\delta F^{\epsilon}}{\delta m}(\cdot, m,\cdot)-\frac{\delta F}{\delta m}(\cdot, m,\cdot)\right|\right|_{C^{\alpha}(\Omega)\times C^{1+\alpha}(\Omega)}=0.
\end{align*}

\noindent
Furthermore,
\begin{align*}
& \text{Lip}_{\alpha}\left(\frac{\delta F^{\epsilon}}{\delta m}\right)\leq C_F,\; \text{Lip}_{1+\alpha}\left(\frac{\delta G^{\epsilon}}{\delta m}\right)\leq C_G,\\
&\int_{\Omega_{\epsilon}}\int_{\Omega_{\epsilon}} \frac{\delta F^{\epsilon}}{\delta m}(x,m,y)\mu(x)\mu(y)\;dxdy\geq 0\quad \text{ and }\\
&\int_{\Omega_{\epsilon}}\int_{\Omega_{\epsilon}} \frac{\delta G^{\epsilon}}{\delta m}(x,m,y)\mu(x)\mu(y)\;dxdy\geq 0, 
\end{align*}
for any signed measure $\mu$ and $m\in\mathcal{P}^{sub}(\Omega)$.
\vspace{2mm}

\item
Finally, we assume the invariance condition in terms of the diffusion matrix and the Hamiltonian $H(x,p)$. Namely, we assume that there exists a $\delta_0>0$ and a $C>0$ such that the following inequality holds :
\begin{equation} \label{PR}
\text{tr}(a(x)D^2d(x))-H_p(x,p)\cdot Dd(x)\geq \frac{a(x)Dd(x)\cdot Dd(x)}{d(x)}-Cd(x)
\end{equation}
for all $p\in\mathbb{R}^d$ and a.e $x\in \Omega\setminus \Omega_{\delta_0}$.

\end{enumerate}

\vspace{5mm}

\begin{rem}
\textbf{(i)} Assumptions (1), (2), (3) and (4) are standard in the literature and are used to prove the existence and the uniqueness of a solution to the MFG system and the linearized system as well as some useful estimates.
\vspace{1mm}

\noindent
\textbf{(ii)} Assumption (5) provides us with compatibility conditions which allow us to get the regularity of the solutions to \textbf{(MFG)} and construct the solution of the linearized system. Its usefulness will be apparent in Proposition \ref{L}, where we prove the existence of a solution to the linearized system.
\vspace{1mm}

\noindent
\textbf{(iii)} Assumption (6) implies that the stochastic trajectories stay inside the domain $\Omega$ almost surely and allows us to ``integrate by parts'' in some specific instances that will be examined later. For more information about this condition, we refer to [1].

\end{rem}

\vspace{7mm}

\addtocontents{toc}{\protect\setcounter{tocdepth}{2}}

\section{Preliminary results}

\subsection{Estimates}
In this section, we will state two estimates from [2] that will be useful in the proofs later.
\vspace{2mm}

\begin{lem} \label{3.1}
Suppose $a$ satisfies Hypothesis (1) and $b,f\in L^{\infty}(Q_T)$. Furthermore, let $\psi\in C^{1+\alpha,N}$ with $0\leq \alpha< 1$. Then, the unique solution $z$ of the problem 
$$\begin{cases}
-\partial_t z-\text{tr}(a(x)D^2z)+b(t,x)\cdot Dz=f(t,x)\text{ },\text{ in }(0,T)\times\Omega,\\
z(T)=\psi\text{ },\;\text{in }\Omega,\\
aDz\cdot\nu|_{\partial\Omega}=0,\text{ on }\partial\Omega,
\end{cases}$$
satisfies 
$$||z||_{\frac{1+\alpha}{2},1+\alpha}\leq C(||f||_{\infty}+||\psi||_{1+\alpha}).$$
\end{lem}

\begin{proof}
This is Lemma 3.1 from [2].
\end{proof}

\vspace{3mm}

\begin{lem} \label{3.2}
Suppose $a$ and $b$ are bounded continuous functions and $\psi\in W^{1,\infty}(\Omega)$. Then, the unique solution $z$ of the problem 
$$\begin{cases}
-\partial_t z-\text{tr}(a(x)D^2z)+b(t,x)\cdot Dz=f(t,x)\text{ },\text{ in }(0,T)\times\Omega,\\
z(T,x)=\psi(x)\text{ },\text{ in }\Omega,\\
aDz\cdot\nu|_{\partial\Omega}=0,\text{ on }\partial\Omega,
\end{cases}$$
satisfies a Hölder condition in $t$ and a Lipschitz condition in $x$, namely, there exists $C>0$ such that 
$$|z(t,x)-z(s,x)|\leq C||\psi||_{W^{1,\infty}}|t-s|^{1/2},\text{   }
|z(t,x)-z(t,y)|\leq C||\psi||_{W^{1,\infty}}|x-y|.$$
\end{lem}

\begin{proof}
This is Lemma 3.2 from [2].
\end{proof}

\vspace{5mm}

\subsection{Properties of the domain and the subdomains}

Since the boundary $\partial\Omega$ of the domain we are working on has $C^{3+\alpha}$-regularity, the boundaries $\partial\Omega_{\epsilon}$, $\epsilon\in (0,\frac{\epsilon_0}{3}]$, of the subdomains inherit some regularity as well. This property is summarized in the following Proposition.

\vspace{4mm}

\begin{prop} \label{domains}
 For all $\epsilon\in [0,\frac{\epsilon_0}{3}]$ the following statements are true :
 
 \noindent
 (i) There exists an $M>0$, independent of $\epsilon$, with the following property: for every $x\in\partial\Omega_{\epsilon}$, there exists an $r_x>0$ such that $\partial\Omega_{\epsilon}\cap B(x,r_x)$ is the graph a  function $\phi_{\epsilon}\in C^{2+\alpha}$ with
 $$||\phi_{\epsilon}||_{2+\alpha}\leq M.$$
 In particular, $\partial\Omega_{\epsilon}$ is of class $C^{2+\alpha}$.
 
 \vspace{1mm}
 
 \noindent
 (ii) $\Omega$ and $\Omega_{\epsilon}$ both satisfy the interior ball condition with a radius that can be chosen to be independent of $\epsilon$.

\end{prop}

\begin{proof}
We note that for $\epsilon=0$, the second result is well known (e.g [7]), while the first result follows from the compactness of $\partial\Omega$. So, there exists $M_0>0$, such that for all $x\in\partial\Omega$, there exists an $r_x>0$ such that $\partial\Omega\cap B(x,r_x)$ is the graph a  function $\phi\in C^{3+\alpha}$ with
 $$||\phi||_{3+\alpha}\leq M_0\;.$$
 In particular, $||\phi||_{2+\alpha}\leq M_0\;.$

\vspace{1mm}
\noindent
We will start by proving the first assertion. Let $x\in \partial\Omega_{\epsilon}$. Since $\epsilon\in(0,\frac{\epsilon_0}{3}]$, there exists a unique $\tilde{x}\in\partial\Omega$ such that $d(x)=|x-\tilde{x}|$. By the $C^{3+\alpha}$-regularity of $\partial\Omega$, there exists an $r>0$ and a function $\phi\in C^{3+\alpha}$ such that $\partial\Omega\cap B(\tilde{x},r)=\{ (x_1,...,x_d)\in B(\tilde{x},r) : x_d=\phi(x_1,...,x_{d-1})\}$. So, $\partial\Omega\cap B(\tilde{x},r)$ is the zero set of  $$g(x_1,...,x_d)=x_d-\phi(x_1,...,x_{d-1})\text{, where }g\in C^{3+\alpha}\text{ with }||\phi||_{3+\alpha}\leq M_0.$$

\noindent
Also, note that, there exists an $r'>0$ such that $B(x,r')\subseteq \Omega \setminus \Omega_{\epsilon_0}$ and for every $y\in B(x,r)\cap \partial \Omega_{\epsilon}$ we have that $y-\epsilon \nu (y)\in B(\tilde{x},r)\cap \partial\Omega$. So, we can define the function
$$g_{\epsilon}:B(x,r')\rightarrow \mathbb{R}\text{ with } g_{\epsilon}(y):=g(y-\epsilon\nu(y)).$$
It is easy to see that $B(x,r')\cap \partial \Omega_{\epsilon}$ is contained in the zero set of $g_{\epsilon}$. In addition, $||g_{\epsilon}||_{2+\alpha}$ can be bounded in terms of $||g||_{2+\alpha}$ and $||\nu||_{2+\alpha}$. Since, $||g||_{2+\alpha}$ can be bounded in terms of $M_0$, independently of $\tilde{x}$, we deduce that $||g_{\epsilon}||_{2+\alpha}\leq M,$ where $M$ depends only $M_0$, $||\nu||_{2+\alpha}$ and the domain $\Omega$ . Finally, by a simple inverse function theorem argument, we can deduce that $B(x,r')\cap \partial \Omega_{\epsilon}$, as the zero set of $g_{\epsilon}$, can be expressed as a graph.
\vspace{2mm}

\noindent
To prove the second statement, let $\epsilon\in(0,\frac{\epsilon_0}{3}]$ and $x\in\partial \Omega_{\epsilon}$. We need to find an $r$ such that the ball $B(x+r\nu(x),r)$ stays inside $\Omega_{\epsilon}$. Again, since $\epsilon\in(0,\frac{\epsilon_0}{3}]$, there exists a unique $\tilde{x}\in\partial\Omega$ such that $d(x)=|x-\tilde{x}|=\epsilon$. Without loss of generality, we may assume that $\tilde{x}=0$, thus $x=\epsilon\nu (0)$ and $x+r\nu(x)=(\epsilon+r)\nu (0)$. We claim that if $r=\frac{\epsilon_0}{3}$, then 
$$B\left(\left(\epsilon+\frac{\epsilon_0}{3}\right)\nu(0),\frac{\epsilon_0}{3}\right)\subset \Omega_{\epsilon}\text{  and  }B\left(\left(\epsilon+\frac{\epsilon_0}{3}\right)\nu(0),\frac{\epsilon_0}{3}\right)\cap \partial\Omega_{\epsilon}=\{x\}=\{\epsilon\nu(0)\}\;,$$
which is exactly what we wanted to show.

\noindent
Indeed, let $y_0\in \partial\Omega_{\epsilon}$. There exists a unique $y\in\partial\Omega$ such that $d_{\Omega}(y_0)=|y-y_0|=\epsilon$. Then, since $\epsilon+\frac{\epsilon_0}{3}<\epsilon_0$,
\begin{align*}
    \left|y_0-(\epsilon+\frac{\epsilon_0}{3})\nu(0)\right|& \geq 
 \left|y-(\epsilon+\frac{\epsilon_0}{3})\nu(0)\right|-|y_0-y| \\
 &\geq \epsilon+\frac{\epsilon_0}{3}-\epsilon=\frac{\epsilon_0}{3}\; ,
 \end{align*}
 with equality if and only if $y_0=x=\epsilon\nu(0)$. Thus, $B\left(\left(\epsilon+\frac{\epsilon_0}{3}\right)\nu(0),\frac{\epsilon_0}{3}\right)$ intersects $\partial\Omega_{\epsilon}$ only at $x=\epsilon\nu(0)$. Finally, it is easy to see that $B\left(\left(\epsilon+\frac{\epsilon_0}{3}\right)\nu(0),\frac{\epsilon_0}{3}\right)\subset \Omega_{\epsilon}\setminus\Omega_{\epsilon_0}\subset \Omega_{\epsilon}$. The result follows.
 \end{proof}

\vspace{7mm}

\noindent
In [2], Ricciardi solves the \textbf{ME} in a bounded domain with Neumann conditions by using Lemmata 3.1, 3.2 and some parabolic estimates from [4] and [6]. Throughout this note we will be using these parabolic estimates for the domains $\Omega_{\epsilon}$, $0<\epsilon<\frac{\epsilon_0}{3}$. Going carefully through their proof, which relies heavily on propositions from [4] and [6], we deduce that the constants in them depend on the domain $\Omega$ in a way that they can be chosen to be independent of $\epsilon$, due to Proposition \ref{domains}. More specifically, we have the following:

\vspace{2mm}

\begin{rem} The constants in the estimates from [2] depend on 
\begin{itemize}
    \item The dimension.
    \item The coefficients of the equations.
    \item The radius in the uniform interior ball condition.
    \item $M$ from Proposition  \ref{domains}.
    \item $\alpha$
\end{itemize}
and the constants are bounded by a polynomial function of the maximum of these quantities.
By Proposition \ref{domains} and our assumptions, all the above can be bounded independently of $\epsilon$, so the constants in the estimates we will be using from [2] are uniformly bounded.

\end{rem}

\vspace{5mm}

\subsection{Extension Theorems}

\noindent
In order to be able to use the estimates from [2], the terminal data of the parabolic problems should satisfy the Neumann compatibility conditions. To achieve that, we are going to make use of Theorem 0.3.2 from [6]. We, hence, state the following result.

\vspace{4mm}
\noindent

\begin{prop} \label{extension}
 Suppose that a domain $\Omega\subseteq \mathbb{R}^n$ has a boundary of class $C^{2+\alpha}$. Then, for any function $f\in C^{1+\alpha}(\partial\Omega)$, there exists a function $\mathcal{N}f\in C^{2+\alpha}(\bar{\Omega})$ such that 
$$\frac{\partial}{\partial( a\nu)}\mathcal{N}f=f$$
i.e,  $aD_x\mathcal{N}f\cdot \nu|_{\partial\Omega}=f$ on $\partial\Omega$. Furthermore, there exists a constant $C$ such that
$$||\mathcal{N}f||_{2+\alpha}\leq C||f||_{1+\alpha}.$$

\end{prop}

\vspace{5mm}
\begin{rem} (i) In the subdomains $\Omega_{\epsilon}$, we will denote the operator given by Proposition \ref{extension} by $\mathcal{N}_{\epsilon}$. 

\noindent
(ii) For any $\epsilon\in [0,\frac{\epsilon_0}{3})$, by the construction of the operator $\mathcal{N}_{\epsilon}$, $C$ will be independent of $\epsilon$ as well.
\end{rem}

\vspace{5mm}

\subsection{Wasserstein distance}
Suppose that $m_1,m_2\in\mathcal{P}(\Omega)$. Later in the paper, we are going to need to compare $\bm{d_1}(m_1,m_2)$ and $\bm{d_1}^{\epsilon}(m_1,m_2)$ (Remark 2.2). A useful estimate is given by the following Proposition
\begin{prop} \label{distances}
Let $m_1,m_2\in\mathcal{P}(\Omega)$ and $\epsilon\in (0,\frac{\epsilon_0}{3})$. Then, there exists a constant $C>0$ that depends on $\Omega$ such that
$$\bm{d_1}^{\epsilon}(m_1,m_2)\leq \bm{d_1}(m_1,m_2)+C(m_1+m_2)(\Omega\setminus\Omega_{\epsilon}).$$
In particular,
$$\limsup_{\epsilon\rightarrow 0}\bm{d_1}^{\epsilon}(m_1,m_2)\leq \bm{d_1}(m_1,m_2).$$
\end{prop}

\begin{proof}
Let $\phi:\Omega_{\epsilon}\rightarrow \mathbb{R}$ be a Lipschitz continuous function with $\text{Lip}(\phi)\leq 1$ and $x_0\in \Omega_{\epsilon}$. By adding a constant, we may assume that $\phi(x_0)=0$.

\noindent
By the Kirszbraun theorem ([15]) we can extend $\phi$ to a function $\tilde{\phi}:\Omega\rightarrow\mathbb{R}$ with $\text{Lip}(\tilde{\phi})=\text{Lip}(\phi)\leq 1$. So, 

$$\left|\int_{\Omega\setminus\Omega_{\epsilon}}\tilde{\phi}(m_1-m_2)\right|\leq \int_{\Omega\setminus\Omega_{\epsilon}}|x-x_0|(m_1+m_2)dx\leq \text{diam}(\Omega)\cdot(m_1+m_2)(\Omega\setminus\Omega_{\epsilon}).$$
Using this estimate we get
$$\int_{\Omega_{\epsilon}}\phi (m_1-m_2)=\int_{\Omega}\tilde{\phi} (m_1-m_2)-\int_{\Omega\setminus\Omega_{\epsilon}}\tilde{\phi} (m_1-m_2)$$
$$\leq \bm{d_1}(m_1,m_2)+\text{diam}(\Omega)\cdot(m_1+m_2)(\Omega\setminus\Omega_{\epsilon}).$$
Taking supremum over $\phi$ we get
$$\bm{d_1}^{\epsilon}(m_1,m_2)\leq\bm{d_1}(m_1,m_2)+\text{diam}(\Omega)\cdot(m_1+m_2)(\Omega\setminus\Omega_{\epsilon}),$$
which is what we wanted to prove.
\end{proof}

\vspace{7mm}

\addtocontents{toc}{\protect\setcounter{tocdepth}{2}}

\section{The Mean Field Game System}

\noindent
In [1], Porretta and Ricciardi prove the existence and the uniqueness of a solution to the mean field game system (MFG for short):
$$\bm{(MFG)}: \begin{cases}
-\partial_tu-\text{tr}(a(x)D^2u)+H(x,Du)=F(x,m(t)),\text{ in }(0,T)\times\Omega,\\
\partial_t m-\text{div}(a(x)Dm)-\text{div}\left(m (H_p(x,Du)+\tilde{b}(x))\right)=0,\text{ in }(0,T)\times\Omega,\\
u(T,x)=G(x,m(T)),\quad m(t_0)=m_0\; ,\\
\end{cases}$$
provided that (\ref{PR}) holds and $m_0\in L^1(\Omega)$, $m_0\geq 0$. Here, we wrote the second equation in divergence form and
$$\tilde{b}_i(x)= \sum_{j=1}^d\frac{\partial a_{ji}}{\partial x_j}(x),\quad\quad i=1,...,d\;,$$
so that $\text{tr}(a(x)D^2u)=\text{div}(a(x)Du)-\tilde{b}(x)\cdot Du$ and $\sum_{i,j}\partial_{i,j}^2(a_{ij}(x)m)=\text{div}(a(x)Dm)+\text{div}(\tilde{b}(x)m)$.

\noindent
In addition, again when $m_0\in L^1(\Omega),m_0\geq 0$, after an easy modification in the last step of their proof of Proposition 4.3 ([1]), we deduce the existence and the uniqueness of a weak solution to the Fokker-Planck equation :
\begin{equation}\label{FP}
    \begin{cases}\partial_tm -\sum_{i,j}\partial^2_{ij}(a_{ij}(x)m)-\text{div}(\alpha m)=0,\text{ in }[t_0,T]\times\Omega ,\\
m(t_0)=m_0,
\end{cases}
\end{equation}
where $\alpha :Q_T\rightarrow\mathbb{R}^d$ is bounded function satisfying
$$\text{tr}(a(x)D^2d(x))-\alpha(x,t)\cdot Dd(x)\geq \frac{a(x)Dd(x)\cdot Dd(x)}{d(x)}-Cd(x)$$
for all $x\in \Omega\setminus \Omega_{\delta_0}$ and $t\in [0,T]$, in the sense of the following definition.
\begin{defn}
We say that $m\in C([0,T];\mathcal{P}(\Omega)) $ is a weak solution to (\ref{FP}) if for every $t\in [0,T]$ and each $\varphi\in C([0,t]; L^1(\Omega))\cap L^{\infty}([0,t]\times\Omega)$ such that $\varphi$ satisfies
$$\begin{cases}
-\partial_t\varphi -\text{tr}(a(x)D^2\varphi)+\alpha\cdot D\varphi\in L^{\infty}([0,t]\times\Omega)\; ,\\
\varphi(t)=\psi\in L^{\infty}(\Omega)\;,
\end{cases}$$
in the sense of distributions, the weak formulation holds:
$$\int_0^t\int_{\Omega}m(-\partial_t\varphi -\text{tr}(a(x)D^2\varphi)+\alpha\cdot D\varphi)\;dxds=\int_{\Omega}m_0\varphi (0)\;dx-\int_{\Omega}m(t)\psi \;dx.$$
\end{defn}

\vspace{7mm}

\noindent
In this section, we will generalize these results for initial measures $m_0\in \mathcal{P}(\Omega)$ and, using our assumptions on $F,G$ and $H$, we will establish some extra regularity properties for the solutions. The existence of a solution to the \textbf{(MFG)} system and its regularity properties, when $m_0\in L^1(\Omega)$, are summarized in the following proposition:

\vspace{4mm}

\begin{prop} \label{MFGL1}
Suppose that Hypotheses 2.3 are satisfied and $m_0\in L^1(\Omega)\cap\mathcal{P}(\Omega)$. There exists a solution $(u,m)\in C^{1+\frac{\alpha}{2},2+\alpha}(Q_T)\times C([0,T];L^1(\Omega))$ of \textbf{(MFG)}, where $u$ satisfies the first equation in the classical sense, while $m$ satisfies the second equation in the sense of Definition 4.1. \end{prop}

\begin{proof}

\noindent
We first assume that $m_0\in L^1(\Omega)$. Let $(u^{\epsilon},m^{\epsilon})$ be a solution of the problem:

\begin{equation} \label{MFGep}
\begin{cases}
-\partial_t u^{\epsilon}-\text{tr}(a(x)D^2u^{\epsilon})+H(x,Du^{\epsilon})=F(x,m^{\epsilon})\quad ,\text{ in } (0,T)\times \Omega_{\epsilon},\\
\partial_t m^{\epsilon}-\text{div}(a(x)Dm^{\epsilon})-\text{div}\left(m^{\epsilon}(H_p(x,Du^{\epsilon})+\tilde{b}(x))\right)=0\quad ,\text{ in } (0,T)\times\Omega_{\epsilon},\\
m^{\epsilon}(0)=m_0 ,\quad u^{\epsilon}(T,x)=G(x,m^{\epsilon}(T))-\mathcal{N}_{\epsilon}(a(\cdot)D_xG(\cdot,m^{\epsilon}(T)))(x),\\
a(x)Du^{\epsilon}\cdot \nu(x)|_{\partial \Omega_{\epsilon}}=0, \quad \left(a(x)Dm^{\epsilon}+m^{\epsilon}(H_p(x,Du^{\epsilon})+\tilde{b}(x))\right)\cdot \nu(x)|_{\partial \Omega_{\epsilon}}=0,
\end{cases};
\end{equation}

\noindent
given by Proposition 3.3 of [2]. We notice that the terminal function $G(x,m^{\epsilon}(T))-\mathcal{N}_{\epsilon}(a(x)D_xG(\cdot,m^{\epsilon}(T)))$ satisfies the compatibility condition on the boundary of $\Omega_{\epsilon}$, so, by Proposition 3.3 of [2], we also have the bound
\begin{align*}
||u^{\epsilon}||_{1+\frac{\alpha}{2},2+\alpha}&\leq C(||F||_{\frac{\alpha}{2},\alpha}+||G-\mathcal{N}_{\epsilon}(a(\cdot)D_xG(\cdot,m^{\epsilon}(T)))||_{2+\alpha})\\
&\leq C(||F||_{\frac{\alpha}{2},\alpha}+||G||_{2+\alpha}+||a(\cdot)D_xG(\cdot,m^{\epsilon}(T))\cdot \nu|_{\partial\Omega_{\epsilon}}||_{1+\alpha}) ,\quad\text{ in }\Omega_{\epsilon},
\end{align*}
\vspace{2mm}
where $C$ comes from Theorem IV.5.3 of [4] and Proposition \ref{extension}, hence it is independent of $\epsilon$ (Proposition \ref{domains} and Remark 3.6). Also, note that
\vspace{1mm}

$$||a(\cdot)D_xG(\cdot,m^{\epsilon}(T))\cdot \nu|_{\partial\Omega_{\epsilon}}||_{1+\alpha}\rightarrow ||a(\cdot)D_xG(\cdot,m(T))\cdot \nu|_{\partial\Omega}||_{1+\alpha}=0\text{ as }\epsilon\rightarrow 0,$$
\vspace{1mm}

\noindent
so $u^{\epsilon}(T,\cdot)\rightarrow G(m(T),\cdot)$ and $||a(\cdot)D_xG(\cdot,m^{\epsilon}(T))\cdot \nu|_{\partial\Omega_{\epsilon}}||_{1+\alpha}$ is uniformly bounded. Thus, $||u^{\epsilon}||_{1+\frac{\alpha}{2},2+\alpha}$ is uniformly bounded and, by Arzela-Ascoli, we have
$$u^{\epsilon}\rightarrow u,\text{ locally uniformly in } C^{1,2}(Q_T),$$
up to subsequence, and $u\in C^{1+\frac{\alpha}{2},2+\alpha}(Q_T)$ with $||u||_{1+\frac{\alpha}{2},2+\alpha}\leq C$. In addition, by Theorem 5.2 in [1], we know, in addition, that as $\epsilon\rightarrow 0$
$$(u^{\epsilon},m^{\epsilon})\rightarrow (u,m)\quad \text{ in } C(0,T; W^{1,2}_{loc}(\Omega))\times C([0,T];L^1(\Omega))\;$$
with $m$ satisfying the second equation of \textbf{(MFG)} in the sense of Definition 4.1. Passing to the limit, as in Theorem 5.2 of [1], we deduce that $u$ is a classical solution to the first equation of the \textbf{(MFG)} system with the claimed regularity, while $m$ is a solution to the second equation in the sense of Definition \ref{FP}.
\end{proof}
\vspace{2mm}

\begin{prop} \label{5.6.[2]}
Suppose that Hypotheses 2.3 are satisfied. Let $(u_1,m_1)$ and $(u_2,m_2)$ be two solutions of \textbf{(MFG)} with initial conditions $m_{01},m_{02}\in L^1(\Omega)\cap\mathcal{P}(\Omega)$. Then there exists a constant $C>0$ such that 
$$||m_1-m_2||_{L^p(Q_T)}\leq C\bm{d_1}(m_{01},m_{02})\;.$$
\end{prop}

\begin{proof}
Note that $m_{01},m_{02}\in L^1(\Omega)\subset C^{-(1+\alpha)}(\Omega)$. Consider $(u_1^{\epsilon},m_1^{\epsilon})$ and $(u_2^{\epsilon},m_2^{\epsilon})$ the solutions to (\ref{MFGep}) with initial conditions $m_{01},m_{02}$, respectively. Then, by Corollary 5.6 of [2], we have
\begin{equation} \label{5.6.1}
||m_1^{\epsilon}-m_2^{\epsilon}||_{L^p(Q_T^{\epsilon})}\leq C\bm{d_1}^{\epsilon}(m_{01},m_{02})
\end{equation}
and
\begin{equation}\label{5.6.2}
    ||m_i^{\epsilon}||_{L^p}\leq C||m_{0i}||_{C^{-(1+\alpha)}(\Omega_{\epsilon})}\leq C||m_{0i}||_{L^1(\Omega_{\epsilon})}\leq C\text{ for }i\in\{1,2\}.
\end{equation}
\vspace{1mm}

\noindent
It follows from (\ref{5.6.2}) that there exist $\widetilde{m_1},\widetilde{m_2}\in L^p(Q_T)$ such that 
$$m_i^{\epsilon}\rightharpoonup \widetilde{m_i}\text{ for }i\in \{1,2\},$$
weakly in $L^p$, up to subsequences. We claim that $\widetilde{m_1}=m_1$ and $\widetilde{m_2}=m_2$. Indeed, let $\phi:= \text{sign}(m_1-\widetilde{m_1})\in L^{\infty}(Q_T)$. Then, from the different types of convergences we have
$$\int_0^T\int_{\Omega_{\epsilon}}m_1^{\epsilon}\phi dxdt\rightarrow \int_0^T\int_{\Omega}m_1\phi dxdt$$
and
$$\int_0^T\int_{\Omega_{\epsilon}}m_1^{\epsilon}\phi dxdt\rightarrow \int_0^T\int_{\Omega}\widetilde{m_1}\phi dxdt.$$
Thus,
$$\int_0^T\int_{\Omega}m_1\phi dxdt=\int_0^T\int_{\Omega}\widetilde{m_1}\phi dxdt,$$
which, due to the choice of $\phi$, implies
$$\int_0^T\int_{\Omega}|m_1-\widetilde{m_1}|dxdt=0\Rightarrow \widetilde{m_1}=m_1$$
and, similarly, $\widetilde{m_2}=m_2$.
\vspace{1mm}

\noindent
We pass to the limit in (\ref{5.6.1}) and, because of Proposition \ref{distances} and the weak lower semicontinuity of the $L^p$-norm, it follows that
$$||m_1-m_2||_{L^p(Q_T)}\leq C\bm{d_1}(m_{01},m_{02}),$$
which is what we wanted to prove.
\end{proof}

\vspace{3mm}

\begin{thm}\label{MFG sol}
Assume Hypotheses 2.3 and $m_0\in \mathcal{P}(\Omega)$. Then, there exists a solution $(u,m)\in C^{1+\frac{\alpha}{2},2+\alpha}\times C([0,T];\mathcal{P}(\Omega))$ of \textbf{(MFG)}, where $u$ satisfies the first equation in the classical sense, while $m$ satisfies the second equation in the sense of Definition 4.1. Furthermore, if $m_0\in C^{2+\alpha}(\Omega)$ with compact support, then $m\in C^{1+\frac{\alpha}{2},2+\alpha}$.
\end{thm}

\begin{proof}
\textit{Step 1:} Note that, by density, there exists a sequence $(m_{0,n})_{n\in\mathbb{N}}\in L^1(\Omega)$ such that $m_{0,n}\rightarrow m_0$ in $\mathcal{P}(\Omega)$ and let $(u_n,m_n)$ be the unique solution of 
$$
\begin{cases}
-\partial_tu_n-\text{tr}(a(x)D^2u_n)+H(x,Du_n)=F(x,m_n(t)),\text{ in }(0,T)\times\Omega,\\
\partial_t m_n-\text{div}(a(x)Dm_n)-\text{div}\left(m_n (H_p(x,Du_n)+\tilde{b}(x))\right)=0,\text{ in }(0,T)\times\Omega\\
u_n(T,x)=G(x,m_n(T)),\quad m(t_0)=m_{0,n},\\
\end{cases};$$
given by Theorem 5.1 of [1]. We will show that the pair $(u_n,m_n)$ converges to $(u,m)$, which will be a solution of \textbf{(MFG)}.

\noindent
By Proposition \ref{MFGL1}, we have that there exists a constant $C>0$ such that 
$$||u_n||_{1+\frac{\alpha}{2},2+\alpha}\leq C\text{ for all }n\in\mathbb{N}.$$
Thus, by Arzela-Ascoli, there exists a subsequence (still denoted by $(u_n)_{n\in\mathbb{N}}$) such that $u_n\rightarrow u$ uniformly in $C^{1,2}$ and $u\in C^{1+\frac{\alpha}{2},2+\alpha}(Q_T)$.

\noindent
Now, fix $\psi\in W^{1,\infty}(\Omega)$ with $||\psi||_{W^{1,\infty}}\leq 1$ and $s\in (0,T]$. By Proposition 3.4 of [1], there exists $\phi^s_{\epsilon}$ that solves
$$\begin{cases}
-\partial_t \phi^s_{\epsilon}-\text{tr}(a(x)D^2\phi^s_{\epsilon})+H_p(x,Du_n)\cdot D\phi^s_{\epsilon}=0,\quad\text{ in } (0,s)\times \Omega_{\epsilon},\\
\phi^s_{\epsilon}(s)=\psi,\text{ in }\Omega,\\
a(x)D\phi^s_{\epsilon}\cdot\nu(x)|_{\partial\Omega_{\epsilon}}=0,
\end{cases}$$
such that $\phi^s_{\epsilon}\rightarrow \phi^s_n$ in $C([0,s]; L^p(\Omega))$ as $\epsilon\rightarrow 0$, where $\phi^s_n$ is a weak solution to the equation
 $$\begin{cases}
    -\partial_t\phi_n^s-\text{tr}(a(x)D^2\phi_n^s )+H_p(x,Du_n)\cdot D\phi_n^s=0,\text{ in }(0,s)\times\Omega,\\
    \phi_n (s)=\psi,\text{ in }\Omega.\\
    \end{cases}$$
However, by Lemma \ref{3.2}, we have
$$|\phi^s_{\epsilon}(t,x)-\phi^s_{\epsilon}(w,y)|\leq C||\psi||_{W^{1,\infty}}(|t-w|^{1/2}+|x-y|),$$
where $C$ is independent of $\epsilon$ due to Proposition \ref{domains}. Thus, by Arzela-Ascoli,  $\phi^s_{\epsilon}\rightarrow \phi^s_n$ uniformly as $\epsilon\rightarrow 0$ and
\begin{equation}
    \sup_{t\in[0,s]}||\phi^s_n(t,\cdot)||_{W^{1,\infty}}\leq C||\psi||_{W^{1,\infty}}\leq C.
\end{equation}
Set $\phi^s_{n,k}:=\phi^s_n-\phi^s_k$. This function satisfies
$$\begin{cases}
-\partial_t \phi^s_{n,k}-\text{tr}(a(x)D^2\phi^s_{n,k})+H_p(x,Du_n)\cdot \phi^s_{n,k}=(H_p(x,Du_k)-H_p(x,Du_n))\cdot D\phi_k,\\
\hspace{12cm}\text{in }(0,s)\times\Omega,\\
\phi^s_{n,k}(s)=0,\text{ in }\Omega.
\end{cases}$$
By Proposition 3.4 of [1], once again, we know that $\phi_{n,k}^s$ can be obtained as a limit of $\phi_{n,k}^{s,\epsilon}$ solving the same equation in $\Omega_{\epsilon}$ with Neumann conditions. Therefore, by Lemma \ref{3.1}, we get that there exists a constant $C>0$, such that
\begin{equation}\sup_{t\in [0,s]}||\phi^s_{n,k}(t,\cdot)||_{1+\alpha}\leq C||(H_p(x,Du_k)-H_p(x,Du_n))\cdot D\phi_k||_{\infty}\leq C||Du_k-Du_n||_{\infty},\end{equation}
where in the last step we used (4.5) and the Lipschitz continuity of $H_p$. 

\noindent
Now, by Definition \ref{FP}, we have
$$\int_{\Omega}\psi (x)m_n(s)dx=\int_{\Omega}\phi^s_n (s,x)m_n(s)dx=\int_{\Omega}\phi^s_n (0,x)m_{0,n}dx.$$
Thus,
\begin{align*}
\int_{\Omega}\psi (x)(m_n(s)-m_k(s)dx & =\int_{\Omega}\phi^s_n (0,x)m_{0,n}dx-\int_{\Omega}\phi^s_k(0,x)m_{0,k}dx\\
&= \int_{\Omega}\phi_n^s(0,x)(m_{0,n}-m_{0,k})dx+\int_{\Omega}(\phi_n^s(0,x)-\phi^s_k(0,x))m_{0,k}dx\\
&\leq C\bm{d_1}(m_{0,n},m_{0,k})+||\phi^s_{n,k}||_{\infty}\\
&\leq  C\left( \bm{d_1}(m_{0,n},m_{0,k})+||Du_k-Du_n||_{\infty}\right).
\end{align*}
In consequence, taking supremum over $\psi\in W^{1,\infty}$ and $s$, we deduce
$$\sup_{s\in[0,T]}\bm{d_1}(m_n(s),m_k(s))\leq C\left( \bm{d_1}(m_{0,n},m_{0,k})+||Du_k-Du_n||_{\infty}\right).$$
Since $(Du_n)_{n\in\mathbb{N}}$ and $(m_{0,n})_{n\in\mathbb{N}}$ are Cauchy, we deduce that $(m_n)_{n\in\mathbb{N}}$ is Cauchy in $C([0,T];\mathcal{P}(\Omega))$. In addition, it follows from Proposition \ref{5.6.[2]} that 
$$||m_n-m_k||_{L^p(Q_T)}\leq C\bm{d_1}(m_{0,n},m_{0,k}).$$
Therefore, $(m_n)_{n\in\mathbb{N}}$ is also Cauchy in $L^p(Q_T)$. Thus, there exists an $m\in C([0,T]; \mathcal{P}(\Omega))\cap L^p(Q_T)$ such that $m_n\rightarrow m$ in $C([0,T]; \mathcal{P}(\Omega))\cap L^p(Q_T)$.
\vspace{1mm}

\noindent
Finally, since
$$(u_n,m_n)\rightarrow (u,m)\text{ in }C^{1,2}\times \left(C([0,T]; \mathcal{P}(\Omega))\cap L^p(Q_T)\right),$$
we deduce that $(u,m)$ solves the \textbf{(MFG)} system in the sense described.
\vspace{2mm}

\noindent
\textit{Step 2.} Finally, we show the last claim. We assume $m_0\in C^{2+\alpha}(\Omega)$ with compact support and $\epsilon$ is such that $\text{spt}(m_0)\subset \Omega_{\epsilon}$. We know from Proposition 3.3 of [2] that $m^{\epsilon}$ satisfies
$$\begin{cases}
\partial_t m^{\epsilon}-\text{tr}(a(x)D^2m^{\epsilon})-m^{\epsilon}\text{div}\left(H_p(x,Du^{\epsilon})+\tilde{b}(x)\right)-Dm^{\epsilon}\cdot \left(2\tilde{b}(x)+H_p(x,Du^{\epsilon})\right)=0,\\
\quad\quad\quad\quad\quad\quad\quad\quad\quad\quad\quad\quad\quad\quad\quad\quad\quad\quad\quad\quad\quad \quad \quad \quad \quad \quad \quad \quad \quad \quad \quad \;\;   \text{ in } (0,T)\times\Omega_{\epsilon},\\
m^{\epsilon}(0)=m_0,\\
\left(a(x)Dm^{\epsilon}+m^{\epsilon}(\tilde{b}(x)+H_p(x,Du^{\epsilon}))\right)\cdot \nu(x)|_{\partial \Omega_{\epsilon}}=0.
\end{cases}$$
Note that $u^{\epsilon}\in C^{1+\frac{\alpha}{2},2+\alpha}$ and $m_0$ satisfies the Neumann compatibility conditions. Therefore, by Schauder theory (Theorem IV. 5.3 of [4]) and the uniform bounds established for $u^{\epsilon}$ in Proposition \ref{MFGL1}, we get that there exists a constant $C$, independent of $\epsilon$, such that 
$$||m^{\epsilon}||_{1+\frac{\alpha}{2},2+\alpha}\leq C\quad ,\text{ in }\Omega_{\epsilon}\;.$$
By Arzela-Ascoli, $m^{\epsilon}\rightarrow m$ locally uniformly in $C^{1,2}$ as $\epsilon\rightarrow 0$ and 
$$||m||_{1+\frac{\alpha}{2},2+\alpha}\leq C.$$
The proof is complete.
\end{proof}

\vspace{7mm}

\noindent
Moreover, \textbf{(MFG)} also satisfies some stability properties. However, before we state and prove the result we are going to need two lemmata. 
\vspace{2mm}

\begin{lem} \label{Lasry-Lions}
Suppose that $m_{01},m_{02}\in\mathcal{P}(\Omega)$ and let $(u_1,m_1)$ and $(u_2,m_2)$ be two solutions of the \textbf{(MFG)} system with $m_1(t_0)=m_{01}$ and $m_2(t_0)=m_{02}$. Then,
$$\int_{t_0}^T\int_{\Omega}(H(x,Du_2)-H(x,Du_1)-H_p(x,Du_1)\cdot(Du_2-Du_1))m_1(t)dxdt$$
$$+\int_{t_0}^T\int_{\Omega}(H(x,Du_1)-H(x,Du_2)-H_p(x,Du_2)\cdot (Du_1-Du_2))m_2(t)dxdt$$
$$\leq \int_{\Omega}(u_1(t_0,x)-u_2(t_0,x))(m_{01}-m_{02})dx\;.$$
\end{lem}

\begin{proof}
The proof is standard ([12]), so we provide a sketch. Let $w:=u_1-u_2$. Note that for $i\in \{1,2\}$ we have
$$\begin{cases}
-\partial_tw -\text{tr}(a(x)D^2w)+Dw\cdot H_p(x,Du_i)\in L^{\infty}([t_0,T]\times\Omega)\; ,\\
w(T,\cdot))=G(\cdot, m_1(T))-G(\cdot, m_2(T))\in L^{\infty}(\Omega)\;,
\end{cases}$$
so by Definition 4.1 we have 
in the sense of distributions, the weak formulation holds:
\begin{multline*}
\int_{t_0}^T\int_{\Omega}m_i(-\partial_t w -\text{tr}(a(x)D^2w)+Dw\cdot H_p(x,Du_i))\;dxds=\\
\int_{\Omega}m_{0i}w(t_0,x)dx-\int_{\Omega}m_i(T)w(T,x)dx.
\end{multline*}
Using the first equation of \textbf{(MFG)} this rewrites as
\begin{multline*}
\int_{t_0}^T\int_{\Omega}m_i(H(x,Du_2)-H(x,Du_1)+F(x,m_1)-F(x,m_2)+Dw\cdot H_p(x,Du_i))\;dxds=\\
\int_{\Omega}m_{0i}w(t_0,x)dx-\int_{\Omega}m_i(T)w(T,x)dx.
\end{multline*}
The result follows by subtracting the two equations and using the monotonicity of $F$ and $G$.
\end{proof}


\vspace{5mm}

\begin{lem} \label{stability estimates}
Suppose Hypotheses 2.3 are satisfied and that $m_{01},m_{02}\in\mathcal{P}(\Omega)$ and let $(u_1,m_1)$ and $(u_2,m_2)$ be two solutions of the \textbf{(MFG)} system with $m_1(t_0)=m_{01}$ and $m_2(t_0)=m_{02}$. Then, there exists a constant $\tilde{C}>0$, independent of $m_{01}$, $m_{02}$ and $t_0$, such that 
$$\sup_{t\in [t_0,T]}||(u_1-u_2)(t,\cdot)||_{2+\alpha}\leq \tilde{C}\sup_{t\in [t_0,T]}\bm{d_1}(m_1(t),m_2(t))\text{  and }$$
$$\sup_{t\in [t_0,T]}\bm{d_1}(m_1(t),m_2(t))\leq \tilde{C}\bm{d_1}(m_{01},m_{02}).$$
In particular, 
$$\sup_{t\in [0,T]}\sup_{m_1\neq m_2}\left[\bm{d_1}(m_1,m_2)^{-1}||U(t,\cdot,m_1)-U(t,\cdot,m_2)||_{2+\alpha}\right]\leq C$$
and the solution to \textbf{(MFG)} is unique.
\end{lem}

\begin{proof}
We set
\begin{align*}
    & \bar{u}  :=u_1-u_2\;, \\
   &  V(t,x)  =\int_0^1H_p(x,sDu_1+(1-s)Du_2)ds\;,\\
    & f(t,x) =\int_0^1\int_{\Omega}\frac{\delta F}{\delta m}(x,sm_1(t)+(1-s)m_2(t),y)(m_1(t)-m_2(t))dyds\quad\text{   and }\\
    & g(x) =\int_0^1\int_{\Omega}\frac{\delta G}{\delta m}(x,sm_1(T)+(1-s)m_2(T),y)(m_1(T)-m_2(T))dyds.
\end{align*}
Then, by Theorem 3.8 of [1], $\bar{u}$ is the unique solution of the equation

\begin{equation}\begin{cases}
-\partial_t \bar{u}-\text{tr}( a(x)D^2\bar{u})+V(t,x)\cdot D\bar{u}=f(t,x),\text{ in }(t_0,T)\times \Omega,\\
\bar{u}(T,x)=g(x),\text{ in }\Omega,
\end{cases}\end{equation}
where near the boundary we assume $(2.1)$. Also, let $\bar{u}^{\epsilon}$ be the solution of the equation 
$$\begin{cases}
-\partial_t \bar{u}^{\epsilon}-\text{tr}( a(x)D^2\bar{u}^{\epsilon})+V(t,x)\cdot D\bar{u}^{\epsilon}=f(t,x)\; ,\text{ in }(t_0,T)\times \Omega_{\epsilon},\\
\bar{u}^{\epsilon}(T,x)=g(x)-\mathcal{N}_{\epsilon}(a(\cdot))D_xg(\cdot))\;, \text{ in }\Omega_{\epsilon},\\
a(x)D\bar{u}^{\epsilon}\cdot \nu(x)|_{\partial\Omega_{\epsilon}}=0.
\end{cases}$$

\noindent
Using the hypotheses on $H_p$, $F$, $G$ and the regularity of $\bar{u}=u_1-u_2$, we can observe that $V$ is bounded in $C^{\frac{\alpha}{2},\alpha}$,  $||f||_{0,\alpha}$ is bounded by  $C\sup_{t\in [t_0,T]}\bm{d_1}(m_1(t),m_2(t))$ and that $||g||_{2+\alpha}$ and $||\mathcal{N}_{\epsilon}(a(\cdot)D_xg(\cdot))||_{2+\alpha}$ are bounded by $C\sup_{t\in [t_0,T]}\bm{d_1}(m_1(t),m_2(t))$. Thus, by Theorem 5.1.21 of [6], we deduce that
\begin{equation}\label{yy}
||\bar{u}^{\epsilon}||_{1,2+\alpha}\leq \tilde{C}\sup_{t\in [t_0,T]}\bm{d_1}(m_1(t),m_2(t)),\text{ in }\Omega_{\epsilon},
\end{equation}
where $\tilde{C}$ does not depend on $\epsilon$ due to Remark 3.4. By Arzela-Ascoli, we have that $\bar{u}^{\epsilon}$ converges, up to subsequence, in $C([t_0,T];C^2(K))$ for every compact subset $K$ of $\Omega$, to a solution of $(4.7)$ as $\epsilon\rightarrow 0$. In consequence, by uniqueness,
$$\bar{u}^{\epsilon}\rightarrow \bar{u}\text{ in }C^2(K)$$
and, passing to the limit in (\ref{yy}),
$$\sup_{t\in[t_0,T]}||\bar{u}(t,\cdot)||_{2+\alpha}\leq \tilde{C}\sup_{t\in [t_0,T]}\bm{d_1}(m_1(t),m_2(t)),\text{ in }\Omega,$$
which is the first estimate.

\vspace{3mm}

\noindent
To prove the second estimate we argue as follows. By Lemma \ref{Lasry-Lions} and the hypotheses on $H$, we have
$$\int_{t_0}^T\int_{\Omega}|Du_1-Du_2|^2(m_1(t)+m_2(t))dxdt\leq C\int_{\Omega}(u_1(t_0,x)-u_2(t_0,x))(m_{01}-m_{02})dx.$$
However, by the uniform gradient bound of $u_1,u_2$ proved in Proposition \ref{MFGL1} and Theorem \ref{MFG sol}, the right-hand side is bounded by $C||u_1-u_2||_{2+\alpha}\bm{d_1}(m_{01},m_{02})$. Hence

\begin{equation}\int_{t_0}^T\int_{\Omega}|Du_1-Du_2|^2(m_1(t)+m_2(t))dxdt\leq C||u_1-u_2||_{2+\alpha}\bm{d_1}(m_{01},m_{02})\;.\end{equation}
\vspace{1mm}

\noindent
Now, pick $s\in (t_0,T]$ and let $\phi$ be a function satisfying
$$\begin{cases}
    -\partial_t\phi-\text{tr}(a(x)D^2\phi )+H_p(x,Du_1)\cdot D\phi=0,\text{ in }(t_0,s)\times\Omega,\\
    \phi (s)=\psi,\text{ in }\Omega,\\
    \end{cases}$$
where near the boundary we assume $(2.1)$ and where $\psi\in W^{1,\infty}(\Omega)$. Note that, due to Proposition 3.4 of [1] and Lemma \ref{3.2}, $\phi$ is Lipschitz, because $\phi$ can be obtained as the limit of a uniformly Lipschitz family of functions. Thus, $\phi$ can be used as a test function for both $m_1$ and $m_2$.\\
By subtracting the weak formulation of $m_1,m_2$ we discover
\begin{multline*}
\int_{\Omega}\phi(s,x)(m_1(s)-m_2(s))dx+\int_{t_0}^s\int_{\Omega}(H_p(x,Du_1)-H_p(x,Du_2))\cdot D\phi\text{ } m_2(t)dxdt\\
+\int_{t_0}^s\int_{\Omega}\left(-\partial_t\phi-\text{tr}(a(x)D^2\phi )+H_p(x,Du_1)\cdot D\phi\right)(m_1(t)-m_2(t))dxdt=\\
\int_{\Omega}\phi(t_0,x)(m_{01}-m_{02})dx\;.
\end{multline*}
Using the Lipschitz continuity of $H_p$ with respect to $p$, we get
$$\int_{\Omega}\psi(x)(m_1(s)-m_2(s))dx\leq C\int_{t_0}^s\int_{\Omega}|Du_1-Du_2|m_2(t)dxdt+C\bm{d_1}(m_{01},m_{02}).$$
Now, by Cauchy-Schwarz, we get
$$\int_{\Omega}\psi(x)(m_1(s)-m_2(s))dx\leq C\left(\int_{t_0}^s\int_{\Omega}|Du_1-Du_2|^2m_2(t)dxdt\right)^{1/2}+C\bm{d_1}(m_{01},m_{02})$$
$$\stackrel{(4.9)}{\leq} C(\sup_{t\in[t_0,T]}||u_1(t,\cdot)-u_2(t,\cdot)||_{2+\alpha})^{1/2}\bm{d_1}(m_{01},m_{02})^{1/2}+\bm{d_1}(m_{01},m_{02}))\;.$$
Considering the supremum over all $\psi\in W^{1,\infty}$, we derive
$$\sup_{t\in [t_0,T]}\bm{d_1}(m_1(t),m_2(t))\leq C\left(\sup_{t\in[t_0,T]}||u_1(t,\cdot)-u_2(t,\cdot)||_{2+\alpha})^{1/2}\bm{d_1}(m_{01},m_{02})^{1/2}+\bm{d_1}(m_{01},m_{02})\right).$$
Applying the first estimate, we deduce that
$$\sup_{t\in [t_0,T]}\bm{d_1}(m_1(t),m_2(t))\leq C\left(\sup_{t\in [t_0,T]}\bm{d_1}(m_1(t),m_2(t))^{1/2}\bm{d_1}(m_{01},m_{02})^{1/2}+\bm{d_1}(m_{01},m_{02})\right).$$
Now, by the Arithmetic-Geometric mean inequality, we can easily get
\vspace{1mm}
$$\sup_{t\in [t_0,T]}\bm{d_1}(m_1(t),m_2(t))\leq C\bm{d_1}(m_{01},m_{02}),$$
which is the second estimate. 
\end{proof}

\vspace{5mm}

\begin{cor} \label{stability}(Stability of Solutions)

\noindent
Fix $t_0\in[0,T]$ and let $(m_{0,n})_{n\in\mathbb{N}}\in\mathcal{P}(\Omega)$ such that $\bm{d}_1(m_{0,n},m_0)\rightarrow 0$. Then, there is a convergence of the corresponding solutions to \textbf{(MFG)}
$$(u_n,m_n)\rightarrow (u,m)\text{ in } C([t_0,T];C^{2+\alpha}(\Omega)\times C([t_0,T]; \mathcal{P}(\Omega)).$$ \end{cor}
\vspace{2mm}

\begin{proof}
The result follows from Lemma \ref{stability estimates}. Indeed, we have
$$\sup_{t\in[t_0,T]}||u_n(t,\cdot)-u(t,\cdot)||_{2+\alpha}\leq C\bm{d_1}(m_{0,n},m_{0})\rightarrow 0$$
and, similarly,
$$\sup_{t\in [t_0,T]}\bm{d_1}(m_n(t),m(t))\leq C\bm{d_1}(m_{0,n},m_{0})\rightarrow 0.$$
\end{proof}

\vspace{3mm}

\section{The Linearized System}

\noindent
Before verifying that $U$, defined in $(1.1)$, is a solution to the \textbf{ME}, it is essential to establish some further regularity properties for the solution of \textbf{(MFG)}. In particular, we need to establish the differentiability with respect to the measure variable. This can be done by studying the linearized system. In what follows we use the notation
$$\frac{\delta F}{\delta m}(x,m(t))(\mu (t))=:\left\langle  \frac{\delta F}{\delta m}(x,m(t),\cdot ), \mu (t)\right\rangle_{1+\alpha}\;.$$

\subsection{Solution to the system}

Let $t_0\in [0,T)$, $m_0\in \mathcal{P}(\Omega)$ and $(u,m)$ be the solution of \textbf{(MFG)} associated with $(t_0,m_0)$. We consider the linearized version of \textbf{(MFG)} 

\begin{equation*} \label{L}
\bm{(L)} :\begin{cases}
-\partial_t v-\text{tr}(a(x)D^2v)+H_p(x,Du)\cdot Dv=\frac{\delta F}{\delta m}(x,m(t))(\mu (t)),\text{ in } (t_0,T)\times\Omega,\\
\partial_t\mu-\text{div}(a(x)D\mu)-\text{div}\left(\mu (H_p(x,Du)+\tilde{b}(x))\right)-\text{div}(m H_{pp}(x,Du)Dv)=0,\\
\hspace{11.5cm}\text{ in }(t_0,T)\times\Omega,\\
v(T,x)=\frac{\delta G}{\delta m}(x,m(T))(\mu (T))\text{ in }\Omega,\quad \mu (t_0)=\mu _0\in C^{-(1+\alpha)}(\Omega),
\end{cases}
\end{equation*}

\vspace{3mm} 
\noindent
where near the boundary we assume $(2.1)$. For generality, and for reasons that will be apparent later, we will study the more general linearized system:

\begin{equation} \label{GL}
\begin{cases}
-\partial_t v-\text{tr}(a(x)D^2v)+H_p(x,Du)\cdot Dv=\frac{\delta F}{\delta m}(x,m(t))(\mu (t))+h(t,x),\text{ in }(t_0,T)\times\Omega,\\
\partial_t\mu-\text{div}(a(x)D\mu)-\text{div}\left(\mu (H_p(x,Du)+\tilde{b}(x))\right)-\text{div}(m H_{pp}(x,Du)Dv+c)=0,\\
\hspace{11.5cm}\text{ in }(t_0,T)\times \Omega,\\
v(T,x)=\frac{\delta G}{\delta m}(x,m(T))(\mu (T))+v_T(x)\text{ in }\Omega,\quad \mu (t_0)=\mu _0\in C^{-(1+\alpha)}(\Omega),
\end{cases}
\end{equation}

\vspace{3mm} 
\noindent
where near the boundary we assume $(2.1)$ and where
$$v_T\in C^{2+\alpha}_c,\;\; h\in C^{0,\alpha}([t_0,T]\times\Omega),\;\; c\in L^1([t_0,T]\times \Omega).$$
\vspace{3mm}

\noindent
Before presenting the proof (Theorem \ref{generalGL}), we need to explain what it means for $(v,\mu)$ to be a solution of \textbf{(L)} or $(5.1)$ and to analyze the Fokker-Planck equation. Let $p=\frac{d+2}{d+1+\alpha}$.

\begin{defn}
Let $t_0\in [0,T]$, $f\in L^1(W^{-1,\infty})$, $\mu_0\in C^{-(1+\alpha)}$ and $b\in (C(Q_T))^d$ satisfying
\begin{equation*} 
\text{div}(a(x)Dd(x))-b(t,x)\cdot Dd(x)\geq \frac{a(x)Dd(x)\cdot Dd(x)}{d(x)}-Cd(x),
\end{equation*} 
for some constant $C$, for all $t\in [t_0,T]$, $x\in \Omega\setminus \Omega_{\delta_0}$.

\noindent
We say that a function $\mu\in C\left([t_0,T]; C^{-(1+\alpha)}_c(\Omega)\right)\cap L^p([t_0,T]\times\Omega)$ is a weak solution of the problem
\begin{equation}\label{fFP}
\begin{cases}
\partial_t \mu-\text{div}(a(x)D\mu)-\text{div}(\mu b)=f ,\text{ in }(t_0,T)\times\Omega,\\
\mu(t_0)=\mu_0
\end{cases}
\end{equation}
if for all $\psi\in C^{0,\alpha}(Q_T)$, $\xi\in C^{1+\alpha}_c(\Omega)$ and $\phi$ solution of the problem
$$\begin{cases}
-\partial_t\phi -\text{div}(aD\phi)+b\cdot D\phi=\psi,\text{ in }(t_0,t)\times\Omega,\\
\phi (t)=\xi,\text{ in }\Omega,
\end{cases}$$
the following equality holds
$$\langle \mu(t),\xi\rangle +\int_{t_0}^t\int_{\Omega}\mu(s,x)\psi(s,x)dxds=\langle \mu_0,\phi(0,\cdot)\rangle +\int_{t_0}^t\langle f(s),\phi(s,\cdot)\rangle ds,$$
where $\langle \cdot,\cdot\rangle$ denotes the duality between $C^{-(1+\alpha)}_c$ and $C^{1+\alpha}_c$ in the first case, $C^{-(1+\alpha)}$ and $C^{1+\alpha}$ in the second case and $W^{-1,\infty}$ and $W^{1,\infty}$ in the third case.
\end{defn}
\vspace{4mm}

\begin{defn}
We say that $(v,\mu)\in C([t_0,T]; C^{2+\alpha})\times \left( C([t_0,T];C^{-(1+\alpha)}_c(\Omega))\cap L^1([t_0,T]\times\Omega)\right)$ is a solution to $(5.1)$ if
\begin{itemize}
    \item $v$ satisfies the first equation in the sense of distributions.
    \item $\mu$ satisfies the second equation in the sense of Definition 5.1.
\end{itemize}
\end{defn}
\vspace{5mm}

\noindent
Using the same idea as in [1], we can approximate a solution of 
(\ref{GL}) with solutions of the linearized system 

\begin{equation}\label{epsilonGL}
   \begin{cases}
-\partial_t v^ε-\text{tr}(a(x)D^2v^ε)+H_p(x,Du^ε)\cdot Dv^ε=\frac{\delta F^{\epsilon}}{\delta m}(x,m^ε(t))(\mu^ε (t))+h(t,x),\text{ in }(t_0,T)\times\Omega_{\epsilon},\\
\partial_t \mu^ε-\text{div}(a(x)D\mu^ε)-\text{div}\left(\mu^ε (H_p(x,Du^ε)+\tilde{b}(x))\right)-\text{div}(m^{\epsilon} H_{pp}(x,Du^ε)Dv^ε+c)=0,\\
\hspace{11.5cm}\text{ in }(t_0,T)\times \Omega_{\epsilon},\\
v^ε(T,x)= \frac{\delta G^{\epsilon}}{\delta m}(x,m^ε(T))(\mu^ε (T))+v_T(x)-g_{\epsilon}(x)\text{ in }\Omega_{\epsilon},\quad \mu^ε (t_0)=\mu _0,\\
a(x)Dv^ε\cdot \nu(x)|_{\partial Ω_ε}=0,\; \\ (a(x)Dμ^ε+μ^ε(H_p(x,Du^ε)+\tilde{b}(x))+m^εH_{pp}(x,Du^ε)Dv^ε+c)\cdot \nu(x)|_{\partial Ω _ε}=0,
\end{cases}
\end{equation}
\vspace{2mm}

\noindent
where $\frac{\delta F^{\epsilon}}{\delta m}(x,m^{\epsilon}(t),\cdot)\in C^{1+\alpha}_c(\Omega)$ ,$\frac{\delta G^{\epsilon}}{\delta m}(\cdot,m^{\epsilon}(t),\cdot)\in C^{2+\alpha}_c(\Omega)\times C^{2+\alpha}_c(\Omega) ,\;\epsilon\in[0,\epsilon_0/3)$ are given by Hypothesis (5) and 
$$g_{\epsilon}(x):= \mathcal{N}_{\epsilon}(a(\cdot)D_xv_T(\cdot))(x)\;.$$ 
Note that in (\ref{epsilonGL}) we are viewing $\mu_0$ as the restriction $\mu_0|_{C^{1+\alpha}(\Omega_{\epsilon})}$. It follows from Theorem IV.4.1 of [4] that there exists a constant $C$ independent of $\epsilon$, due to Proposition \ref{domains}, such that  $||\mu_0||_{C^{-(1+\alpha)}(\Omega_{\epsilon})}\leq C||\mu_0||_{C^{-(1+\alpha)}(\Omega)}$, so in the estimates below we will be using $||\mu_0||_{C^{-(1+\alpha)}(\Omega)}$ directly.
\vspace{5mm}

\noindent
In [2], it is proved that, under our assumptions for $a,\;F^{\epsilon},\;G^{\epsilon}$ and $H$ (Hypotheses 1,2,3,4,5), (\ref{epsilonGL})  has a unique solution -in the sense of Definition 5.2, even if in Definition 5.1 we had $\xi\in C^{1+\alpha}_c$- $(v^ε,μ^ε)\in C^{1,2+α}([t_0,T]\times\Omega_{\epsilon})\times (C([t_0,T];C^{-(1+α),N}(Ω_ε))\cap L^1([t_0,T]\times Ω^ε))$, satisfying the estimate
$$ ||v^ε||_{C^{1,2+α}([t_0,T]\times\Omega_{\epsilon})}+\sup_{t\in [t_0,T]}||\mu^ε(t)||_{C^{-(1+α),N}(\Omega_{\epsilon})}+||μ^ε||_{L^p([0,T]\times \Omega_{\epsilon})}\leq CM,$$
for some $p>1$ and some constant $C>0$ which, as we mentioned in Remark 3.4, is independent of $\epsilon$, where
$$M=||v_T||_{2+\alpha}+||\mu_0||_{-(1+\alpha)}+||h||_{0,\alpha}+||c||_{L^1}.$$
Since $C^{1+\alpha}_c(\Omega_{\epsilon})$ is a subspace of $C^{1+\alpha,N}(\Omega_{\epsilon})$, the estimate implies that
$$ ||v^ε||_{1,2+α}+\sup_{t\in [t_0,T]}||\mu^ε(t)||_{C^{-(1+α)}_c(\Omega_{\epsilon})}+||μ^ε||_{L^p([t_0,T]\times \Omega_{\epsilon})}\leq CM\text{  in }\Omega_{\epsilon}.$$

\noindent
Now, we extend $\mu^{\epsilon}$ to be $0$ in $\Omega\setminus \Omega_{\epsilon}$ and we denote by $\bar{\mu}_{\epsilon}$ the extension. Then, the previous estimate gives
\begin{equation} \label{5.1}
||v^{\epsilon}||_{1,2+\alpha}+||\bar{μ}^ε||_{L^p([t_0,T]\times \Omega)}+\sup_{t\in [t_0,T]}||\mu^ε(t)||_{C^{-(1+α)}_c(\Omega_{\epsilon})}\leq CM\text{ in }\Omega_{\epsilon}.
\end{equation}

\noindent
Note that, $C^{-(1+α)}_c(\Omega_{\epsilon})$ is a subspace of $C^{-(1+α)}_c(\Omega)$, therefore Hahn-Banach implies that there is an extension $\mu^ε(t)\in C^{-(1+α)}_c(\Omega_{\epsilon})$ of $\tilde{\mu}^{\epsilon}(t)\in C^{-(1+α)}_c(\Omega)$ such that we still have
$$||\tilde{\mu}^{\epsilon}(t)||_{C^{-(1+α)}_c(\Omega)}\leq CM\text{ for all } t\in [t_0,T].$$

\vspace{3mm}

\begin{prop} \label{mu}
Suppose that Hypotheses 2.3 are satisfied and $m_0\in L^1(\Omega)\cap\mathcal{P}(\Omega)$. Τhere exists a $\mu\in C([t_0,T]; C^{-(1+\alpha)}_c(\Omega))\cap L^p([t_0,T]\times\Omega)$ such that 
\begin{align*}
&\tilde{\mu^{\epsilon}}(t)\stackrel{\star}{\rightarrow} \mu(t)\text{ weak-}\star\text{ in }C^{-(1+\alpha)}_c(\Omega)\text{ for all }t\in[t_0,T]\text{ and }\\
&\bar{\mu^{\epsilon}}\rightharpoonup \mu\text{ weakly in }L^p([t_0,T]\times\Omega).
\end{align*}
\end{prop}

\begin{proof} 
Since for any $t\in [t_0,T]$ we have
$$||\tilde{\mu}^{\epsilon}(t)||_{C^{-(1+α)}_c(\Omega)}\leq CM,$$
Banach-Alaoglu theorem and the characterization of $C^{-(1+\alpha)}(\Omega)$ yield the existence of a convergent subsequence in the weak-$\star$ topology to some $\tilde{\mu}(t)\in C^{-(1+\alpha)}_c(\Omega)$. We will show that $\tilde{\mu}(t)$ is continuous with respect to $t$.

\noindent
Let $f^{\epsilon}:=\text{div}(m^{\epsilon}H_{pp}(x,Du^{\epsilon})Dv^{\epsilon})$. We observe that $f^{\epsilon}$ is uniformly bounded in $W^{-1,\infty}(\Omega_{\epsilon})$, because of our assumptions on $H$, the bounds on $v^{\epsilon}$ (relation (\ref{5.1})) and $u^{\epsilon}$ (Proposition \ref{MFGL1}) and the fact that $||m^{\epsilon}||_{L^1}\leq 1$. We denote by $M_0$ the uniform bound, which is independent of $\epsilon$.
\vspace{1mm}

\noindent
We, also, let $\phi\in C^{1+\alpha}_c(\Omega)$ with compact support and let $\epsilon$ such that $\text{spt}(\phi)\subseteq \Omega_{\epsilon}$. Also, for $s\in [t_0,T]$ consider $\psi^s$ to be the unique solution to 
\begin{equation} \label{a1}
\begin{cases}
-\partial_t \psi^s-\text{div}(a(x)D\psi^s)+\left(H_p(x,Du^{\epsilon})+\tilde{b}(x)\right)\cdot D\psi^s=0, \text{ in }(t_0,s)\times \Omega_{\epsilon},\\
\psi^s(s)=\phi, \text{ in }\Omega_{\epsilon},\\
a(x)D\psi^s\cdot\nu(x)|_{\partial\Omega_{\epsilon}}=0.
\end{cases}
\end{equation}
Note that, by Lemma \ref{3.1}, 
\begin{equation} \label{a2}
    ||\psi^{s}||_{\frac{1+\alpha}{2},1+\alpha}\leq C||\phi||_{1+\alpha},
\end{equation}


\noindent
where $C$ is independent of $\epsilon$ and $s$. In addition, by the uniqueness of a solution to (\ref{a1}), if $t_0\leq s_1<s_2\leq T$, then the function 
$$\tilde{\psi}^{s_2}(\cdot,x):=\psi^{s_2}(\cdot+s_2-s_1,x)$$
solves 
\begin{equation}\label{zz}
    \begin{cases}
-\partial_t \tilde{\psi}^{s_2}-\text{div}(a(x)D\tilde{\psi}^{s_2})+\left(H_p(x,Du^{\epsilon}(\cdot+s_2-s_1,x))+\tilde{b}(x)\right)\cdot D\tilde{\psi}^{s_2}=0, \text{ in }(t_0,s_1)\times \Omega_{\epsilon},\\
\tilde{\psi}^{s_2}(s_1)=\phi, \text{ in }\Omega_{\epsilon},\\
a(x)D\tilde{\psi}^{s_2}\cdot\nu(x)|_{\partial\Omega_{\epsilon}}=0.
\end{cases}
\end{equation}
Therefore, combining this with (\ref{a1}) for $s=s_1$ and estimate (\ref{a2}), Lemma \ref{3.1} and the Lipschitz continuity of $H_p$, as we did in Theorem \ref{MFG sol}, we get
$$\sup_{t\in [t_0,s_1]}||\psi^{s_1}(t,\cdot)-\tilde{\psi}^{s_2}(t,\cdot)||_{1+\alpha}\leq C||\phi||_{1+\alpha}|s_2-s_1|^{\alpha/2}.$$
Moreover, in view of (\ref{a2}), we have
\begin{equation*}
\sup_{t\in [t_0,s_1]}||\tilde{\psi}^{s_2}(t,\cdot)-\psi^{s_2}(t,\cdot)||_{1+\alpha}\leq C||\phi||_{1+\alpha}|s_2-s_1|^{\alpha/2}.
\end{equation*}
Thus, combining the last two estimates
\begin{equation}\label{a3}
    \sup_{t\in [t_0,s_1]}||\psi^{s_2}(t,\cdot)-\psi^{s_2}(t,\cdot)||_{1+\alpha}\leq C||\phi||_{1+\alpha}|s_2-s_1|^{\alpha/2}.
\end{equation}
Now, we use $\psi^{s_1},\psi^{s_2}$ as test functions in the second equation of (\ref{epsilonGL}) and we discover
$$\langle \mu^{\epsilon}(s_1),\phi\rangle = \langle \mu_0,\psi^{s_1}(0)\rangle +\int_{t_0}^{s_1}\langle f^{\epsilon}(t),\psi^{s_1}(t)\rangle dt$$
and
$$\langle \mu^{\epsilon}(s_2),\phi\rangle = \langle \mu_0,\psi^{s_2}(0)\rangle +\int_{t_0}^{s_2}\langle f^{\epsilon}(t),\psi^{s_2}(t)\rangle dt.$$
Subtracting these two equalities we discover
\begin{multline*}
\langle \mu^{\epsilon}(s_2)- \mu^{\epsilon}(s_1),\phi\rangle = \langle \mu_0,\psi^{s_2}(t_0)-\psi^{s_1}(t_0)\rangle +\int_{t_0}^{s_1}\langle f^{\epsilon}(t),\psi^{s_2}(t)-\psi^{s_1}(t)\rangle dt\\ +\int_{s_1}^{s_2}\langle f^{\epsilon}(t),\psi^{s_2}(t)\rangle dt.
\end{multline*}
Now, (\ref{a2}), (\ref{a3}) and the fact that $f^{\epsilon}$ is uniformly bounded in $W^{-1,\infty}(\Omega_{\epsilon})$ give
\begin{multline*}
\langle \mu^{\epsilon}(s_2)- \mu^{\epsilon}(s_1),\phi\rangle  \leq \\ C||\phi||_{1+\alpha}\left( ||\mu_0||_{-(1+\alpha)}|s_2-s_1|^{\alpha/2}+\sup_{t\in [t_0,T]}||f^{\epsilon}(t)||_{W^{-1,\infty}}|s_2-s_1|^{\alpha/2}+\int_{s_1}^{s_2}||f^{\epsilon}(t)||_{W^{-1,\infty}}dt \right)\\
\leq ||\phi||_{1+\alpha}\left( ||\mu_0||_{-(1+\alpha)}|s_2-s_1|^{\alpha/2}+M_0|s_2-s_1|^{\alpha/2}+M_0|s_2-s_1| \right).
\end{multline*}
Finally, we let $\epsilon\rightarrow 0$ and we deduce
\begin{multline*}
\langle \tilde{\mu}(s_2)- \tilde{\mu}(s_1),\phi\rangle\leq \\ C||\phi||_{1+\alpha}\left( ||\mu_0||_{-(1+\alpha)}|s_2-s_1|^{\alpha/2}+M_0|s_2-s_1|^{\alpha/2}+M_0|s_2-s_1|\right)\;,
\end{multline*}
for all $\phi\in C^{1+\alpha}_c(\Omega)$. Hence,
\begin{multline*}
||\tilde{\mu}(s_2)- \tilde{\mu}(s_1) ||_{C^{-(1+\alpha)}_c(\Omega)}\leq \\  C||\phi||_{1+\alpha}\left( ||\mu_0||_{-(1+\alpha)}|s_2-s_1|^{\alpha/2}+M_0|s_2-s_1|^{\alpha/2}+M_0|s_2-s_1|\right)
\end{multline*}
and continuity follows.

\vspace{2mm}

\noindent
However, estimate (\ref{5.1}), along with the reflexivity of $L^p$, also implies that there exists a $\bar{\mu}\in L^p$ such that $\bar{\mu}^{\epsilon}\rightharpoonup \bar{\mu}\in L^p$ weakly in $L^p$. We will show that $\bar{\mu}$ is equal to $\tilde{\mu}$ and, as a consequence, we can call $\mu$ the common limit.

\noindent
Indeed, let $\phi\in C^{1+\alpha}(\overline{Q_T})$ with compact support. Then, for any $\epsilon$ small enough such that $\text{spt}(\phi)\subseteq [t_0,T]\times \Omega_{\epsilon}$, we have
$$\int_{t_0}^T\langle \mu^{\epsilon}(t),\phi(t)\rangle \;dt=\int_{t_0}^T\int_{\Omega}\bar{\mu}^{\epsilon}(t)\phi(x,t)\;dxdt$$
Letting $\epsilon\rightarrow 0$ yields
$$\int_{t_0}^T\langle \tilde{\mu}(t),\phi(t)\rangle\; dt=\int_{t_0}^T\int_{\Omega}\bar{\mu}(t,x)\phi(x,t)\;dxdt\;.$$
Thus, $\tilde{\mu}$ can be considered an element of $L^p(Q_T)$ with $\tilde{\mu}=\bar{\mu}$. The proof is complete.

\end{proof}

\vspace{7mm}

\begin{rem} \label{v}
\textbf{(i)} Another consequence of (\ref{5.1}) is that $||v^{\epsilon}||_{C^{1,2+\alpha}(\Omega_{\epsilon})}$ is uniformly bounded, thereupon there exists $v\in C([t_0,T]; C^2(\Omega))$ such that $v^{\epsilon}(t,\cdot) \rightarrow v(t,\cdot )$ locally uniformly, up to a subsequence. We, hence, constructed a pair $(v,\mu)\in C([t_0,T];C^2(\Omega))\times (C([t_0,T];C^{-(1+\alpha)}_c(\Omega))\cap L^p(Q_T))$.
\vspace{1mm}

\noindent
\textbf{(ii)} By (\ref{5.1}), again, $v^{\epsilon}(t,x)$ is Lipschitz with respect to $t$ with the Lipschitz constant being independent of $\epsilon$ . It follows that $v$ is also Lipschitz with a Lispchitz constant depending only on $\Omega$ and $||\mu_0||_{-(1+\alpha)}$.
\end{rem}

\vspace{5mm}

\begin{prop} \label{L}
Suppose that Hypotheses 2.3 are satisfied and $m_0\in L^1(\Omega)\cap\mathcal{P}(\Omega)$. Then, the pair $(v,\mu)$, as constructed in Proposition \ref{mu} and Remark \ref{v} is a solution to (\ref{GL}).
\end{prop}

\begin{proof} Without loss of generality we assume that $t_0=0$.
\vspace{2mm}

\noindent
\textit{Step 1:} We will first check that the boundary conditions are satisfied. It suffices to show that
$$\frac{\delta G^{\epsilon}}{\delta m}(x,m^ε(T))(\mu^ε (T))+v_T(x)-g_{\epsilon}(x)\rightarrow \frac{\delta G}{\delta m}(x,m(T))(\mu (T))+v_T(x)\text{  as }\epsilon\rightarrow 0\; ;$$
almost everywhere.

\noindent
We have
\begin{multline*}\left|\left| \frac{\delta G^{\epsilon}}{\delta m}(x,m^{\epsilon}(T),\cdot)-\frac{\delta G}{\delta m}(x,m(T),\cdot)\right|\right|_{C^{1+\alpha}(\Omega_{\epsilon})}\leq \\ 
\left|\left| \frac{\delta G^{\epsilon}}{\delta m}(x,m^{\epsilon}(T),\cdot)-\frac{\delta G}{\delta m}(x,m^{\epsilon}(T),\cdot)\right|\right|_{C^{1+\alpha}(\Omega_{\epsilon})}
+\left|\left| \frac{\delta G}{\delta m}(x,m^{\epsilon}(T),\cdot)-\frac{\delta G}{\delta m}(x,m(T),\cdot)\right|\right|_{C^{1+\alpha}(\Omega)}.
\end{multline*}
\noindent
Since, $m^{\epsilon}(T)\rightarrow m(T)$ in $L^1(\Omega)$ (Proposition 4.3 of [1]), we have
$$\lim_{\epsilon\rightarrow 0^+}\left|\left| \frac{\delta G}{\delta m}(x,m^{\epsilon}(T),\cdot)-\frac{\delta G}{\delta m}(x,m(T),\cdot)\right|\right|_{C^{1+\alpha}(\Omega)}=0.$$
In addition, again due to the properties of $G$,

\begin{multline*}
\limsup_{\epsilon\rightarrow 0^+}\left|\left| \frac{\delta G^{\epsilon}}{\delta m}(x,m^{\epsilon}(T),\cdot)-\frac{\delta G}{\delta m}(x,m^{\epsilon}(T),\cdot)\right|\right|_{C^{1+\alpha}(\Omega_{\epsilon})}\leq \\ \lim_{\epsilon\rightarrow 0^+}\sup_{m\in\mathcal{P}^{sub}(\Omega)}\left|\left| \frac{\delta G^{\epsilon}}{\delta m}(x,m,\cdot)-\frac{\delta G}{\delta m}(x,m,\cdot)\right|\right|_{C^{1+\alpha}(\Omega_{\epsilon})}=0.
\end{multline*}
Thus, 
$$\lim_{\epsilon\rightarrow 0^+}\left|\left| \frac{\delta G^{\epsilon}}{\delta m}(x,m^{\epsilon}(T),\cdot)-\frac{\delta G}{\delta m}(x,m(T),\cdot)\right|\right|_{C^{1+\alpha}(\Omega_{\epsilon})}=0$$
and hence,
$$\frac{\delta G^{\epsilon}}{\delta m}(x,m^{\epsilon}(T),\cdot)\rightarrow \frac{\delta G}{\delta m}(x,m(T),\cdot)\text{ in }C^{1+\alpha}\text{ as }\epsilon\rightarrow 0.$$
In consequence, combining this with the weak-$\star$ convergence of $\mu^{\epsilon}$ to $\mu$, we get
$$\frac{\delta G^{\epsilon}}{\delta m}(x,m^{\epsilon}(T),\cdot)(\mu^{\epsilon}(T))\rightarrow \frac{\delta G}{\delta m}(x,m(T),\cdot)(\mu(T)).$$

\noindent 
Finally, in view of the properties of $\mathcal{N}_{\epsilon}$, we discover
\begin{align*}
||g_{\epsilon}||_{2+\alpha}&\leq C||a(\cdot)D_xv_T(\cdot)||_{C^{1+\alpha}(\partial\Omega_{\epsilon})},
\end{align*}
where the constant $C$ does not depend on $\epsilon$. Letting $\epsilon\rightarrow 0$ we get, by the properties of $G$ and $v_T$, that the right-hand side converges to
$$C||a(\cdot)D_xv_T(\cdot)||_{C^{1+\alpha}(\partial\Omega)}=0,$$
so $g_{\epsilon}\rightarrow 0$. The result follows.
\vspace{2mm}

\noindent
\textit{Step 2:} We will prove that the pair $(v,\mu)$ satisfies the two equations of the system separately. Let $\phi\in C^{\infty}_c((0,T]\times \Omega)$. For $\epsilon $ close to $0$, such that the support of $\phi$ is contained in $(0,T]\times \Omega_{\epsilon}$, $\phi$ is an admissible test function for (\ref{epsilonGL}). Integrating the first equation against $\phi$ and integrating by parts, we get
\begin{equation}
    \begin{split}
\int_0^T\int_{\Omega_{\epsilon}}v^{\epsilon}(\partial_t \phi)\;dxdt-
\int_0^T\int_{\Omega_{\epsilon}}\text{tr}(a(x)D^2v^{\epsilon})\phi\; dxdt+\int_0^T\int_{\Omega_{\epsilon}}H_p(x,Du^{\epsilon})\cdot Dv^{\epsilon}\phi \;dxdt=\\ \int_{0}^T\int_{\Omega_{\epsilon}}h(t,x)\phi(t,x)\; dxdt+ \int_{\Omega_{\epsilon}}\int_0^T\int_{\Omega_{\epsilon}}\frac{\delta F^{\epsilon}}{\delta m}(x,m^{\epsilon}(t),y)\phi(x,t)\mu^{\epsilon}(t,y)\;dydtdx\\
+\int_{\Omega_{\epsilon}}\int_0^T\int_{\Omega_{\epsilon}}\left(\frac{\delta G^{\epsilon}}{\delta m}(x,m^{\epsilon}(t),y)\mu^{\epsilon}(y,t)+v_T(x)-g_{\epsilon}(x)\right)\phi \; dydtdx\;.
\end{split}
\end{equation}
It is clear that, due to the convergence of the $v^{\epsilon}$ to $v$ and of $u^{\epsilon}$ to $u$, as $\epsilon\rightarrow 0$, that the left-hand side of $(5.9)$, by the dominated convergence theorem, converges to
$$\int_0^T\int_{\Omega}v(\partial_t \phi)\;dxdt-
\int_0^T\int_{\Omega}\text{tr}(a(x)D^2v)\phi \;dxdt+\int_0^T\int_{\Omega}H_p(x,Du)\cdot (Dv)\phi\; dxdt.$$
On the other hand, the functions
$$h^{\epsilon}(x):=\int_0^T\int_{\Omega_{\epsilon}}\frac{\delta F^{\epsilon}}{\delta m}(x,m^{\epsilon}(t),y)\phi(x,t)\mu^{\epsilon}(t,y)\;dydt$$
are uniformly bounded, because of the $L^p-$bound for $\mu^{\epsilon}$ and the properties of $F$. Since $m^{\epsilon}\rightarrow m$ in $C([0,T]; L^1(\Omega))$ (Proposition 4.3 of [1]) and $\mu^{\epsilon}\rightharpoonup \mu$ weakly in $L^p$, we also get that
$$h^{\epsilon}(x)\rightarrow \int_0^T\int_{\Omega}\frac{\delta F}{\delta m}(x,m(t),y)\phi(x,t)\mu(t,y)\;dydt\text{  a.e}.$$
Finally, by our arguments in Step 1 and the dominated convergence theorem, as $\epsilon\rightarrow 0$
\begin{multline*}
\int_{\Omega_{\epsilon}}\int_0^T\int_{\Omega_{\epsilon}}\left(\frac{\delta G^{\epsilon}}{\delta m}(x,m^{\epsilon}(t),y)\mu^{\epsilon}(y,t)+v_T(x)-g_{\epsilon}(x)\right)\phi \; dydtdx\rightarrow \\ \int_{\Omega}\int_0^T\int_{\Omega}\left(\frac{\delta G}{\delta m}(x,m(t),y)\mu(y,t)+v_T(x)\right)\phi \; dydtdx\;.
\end{multline*}
Thus, again by the dominated convergence theorem, we deduce that the right-hand side of $(5.9)$ converges to 
\begin{multline*}
    \int_{0}^T\int_{\Omega}h(t,x)\phi(t,x)\; dxdt+ \int_{\Omega}\int_0^T\int_{\Omega}\frac{\delta F}{\delta m}(x,m(t),y)\phi(x,t)\mu(t,y)\;dydtdx\\
+\int_{\Omega}\int_0^T\int_{\Omega}\left(\frac{\delta G}{\delta m}(x,m(t),y)\mu(y,t)+v_T(x)\right)\phi \; dydtdx\;.
\end{multline*}
We conclude that $(v,\mu)$ solves the first equation of the linearized system.
\vspace{3mm}

\noindent
To prove that the second equation is satisfied, let $t\in (0,T)$, $f(s,x)\in C^{0,\alpha}(Q_T)$ and $\phi$ be such that
\begin{equation}\label{eqtest}
\begin{cases}
-\partial_s\phi-\text{tr}(a(x)D^2\phi)+H_p(x,Du)\cdot D\phi=f(s,x),\text{ in }(0,t)\times\Omega,\\
\phi(t)=\xi,\text{ in }\Omega,
\end{cases}
\end{equation}
where $\xi\in C^{1+\alpha}_c(\Omega)$. By definition, we know that there exists a sequence $\xi^{\epsilon}\in C^{1+\alpha}_c(\Omega)$ such that $\xi^{\epsilon}\rightarrow \xi$ as $\epsilon\rightarrow 0$ in $C^{1+\alpha}$ with $\text{spt}(\xi^{\epsilon})\subset \Omega_{\epsilon}$.

\noindent
Also, let $\phi^{\epsilon}$ be the solution of
$$\begin{cases}
-\partial_t\phi^{\epsilon}-\text{tr}(a(x)D^2\phi^{\epsilon})+H_p(x,Du^{\epsilon})\cdot D\phi^{\epsilon}=
-\partial_t\phi-\text{tr}(a(x)D^2\phi)+H_p(x,Du)\cdot D\phi,\\
\hfill \hfill \hfill\hfill\hfill\hfill \text{ in }(0,t)\times\Omega_{\epsilon},\\
\phi^{\epsilon}(t)=\xi^{\epsilon},\text{ in }\Omega_{\epsilon},\\
a(x)D\phi\cdot \nu(x)|_{\partial\Omega_{\epsilon}}=0\;.
\end{cases}.$$
Testing the second equation of (\ref{epsilonGL}) with $\phi^{\epsilon}$ yields
\begin{equation}\label{xx}
\begin{split}
\int_0^t\int_{\Omega_{\epsilon}}\mu^{\epsilon}(-\partial_s\phi-\text{tr}(a(x)D^2\phi)+H_p(x,Du)\cdot D\phi)\;dxds-\langle \mu_0,\phi^{\epsilon}(0)\rangle=\\
-\int_0^t\int_{\Omega_{\epsilon}}c\cdot D\phi^{\epsilon}\;dxds
-\int_0^t\int_{\Omega_{\epsilon}}m^{\epsilon}H_{pp}(x,Du^{\epsilon})Dv^{\epsilon}\cdot D\phi^{\epsilon}\;dxds-\langle \mu^{\epsilon}(t),\xi^{\epsilon}\rangle.
\end{split}
\end{equation}

\noindent
Note that, since $H_p(x,Du^{\epsilon})\in C^{\frac{\alpha}{2},\alpha}$, Lemma \ref{3.1} implies
$$||\phi^{\epsilon}||_{\frac{1+\alpha}{2},1+\alpha}\leq C(||\xi^{\epsilon}||_{1+\alpha}+||f||_{\infty}),\text{ in }(0,t)\times\Omega_{\epsilon},$$
where $C$ is independent of $\epsilon$. In addition, since $\phi^{\epsilon}$ is Hölder continuous, looking at the equation satisfied by $\tilde{\phi^{\epsilon}}(s,x)=(t-s)\phi^{\epsilon}(s,x)$, Theorem 5.1.21 of [6] implies that 
$||\phi^{\epsilon}(s,\cdot)||_{C^{2+\alpha}(\Omega_{\epsilon})}$ is uniformly bounded; independently of $\epsilon$, for any $s\in [0,t)$. So, by Arzela-Ascoli, since $\phi^{\epsilon}\rightarrow \phi$ in $C(0,T; L^p(\Omega))$ (Proposition 3.4 of [1]), we get 
\begin{multline}\label{result}
\phi^{\epsilon}(s,\cdot)\rightarrow \phi(s,\cdot) \text{ locally uniformly in } C^{0,1}\text{ for }s\in [0,t],\;\\ \phi^{\epsilon}(s,\cdot)\rightarrow \phi(s,\cdot)\text{ locally uniformly in } C^{2},\text{ for any }s\in [0,t)\\
||\phi||_{\frac{1+\alpha}{2},1+\alpha}\leq C(||\xi||_{1+\alpha}+||f||_{\infty})\text{ and } ||\phi(s,\cdot)||_{2+\alpha}\leq \frac{C}{t-s}.
\end{multline}
In addition, we know that $Du^{\epsilon}(s,\cdot), Dv^{\epsilon}(s,\cdot)$ are uniformly bounded and converge to $Du, Dv$ a.e, respectively, hence we deduce that 
$$H_{pp}(x,Du^{\epsilon})Dv^{\epsilon}\cdot D\phi^{\epsilon}\rightarrow H_{pp}(x,Du)Dv\cdot D\phi\text{ in }L^p.$$
Combining this with the fact that $m^{\epsilon}\rightharpoonup m$ weakly in $L^p$ (by Proposition \ref{5.6.[2]}), we get
$$-\int_0^t\int_{\Omega_{\epsilon}}m^{\epsilon}H_{pp}(x,Du^{\epsilon})Dv^{\epsilon}\cdot D\phi^{\epsilon}\;dxds\rightarrow -\int_0^t\int_{\Omega_{\epsilon}}mH_{pp}(x,Du)Dv\cdot D\phi \;dxds.$$
Finally, by the weak-$(\star)$ convergence,
$$\langle \mu^{\epsilon}(t),\xi^{\epsilon}\rangle\rightarrow \langle \mu(t),\xi\rangle.$$
Passing to the limit in (\ref{xx}), having the above observations in mind, we deduce
$$\int_0^t\int_{\Omega}\mu(-\partial_t\phi-\text{tr}(a(x)D^2\phi)+H_p(x,Du)\cdot D\phi)\;dxds-\langle \mu_0,\phi(0)\rangle=$$
$$-\int_0^t\int_{\Omega}c\cdot D\phi\;dxds-\int_0^t\int_{\Omega}mH_{pp}(x,Du)Dv\cdot D\phi\; dxds-\langle \mu(t),\xi\rangle.$$
Thus, by definition, $(v,\mu)$ satisfies the second equation of (\ref{GL}). 
\end{proof}

\vspace{5mm}

\begin{prop} \label{es3}
Let $m_0\in L^1(\Omega)\cap\mathcal{P}(\Omega)$. Then, $(v,\mu)$, constructed in Proposition 5.3 and Remark 5.4, satisfies the estimate
$$\sup_{t\in [t_0,T]}||v(t,\cdot)||_{2+\alpha}+||\mu||_{L^p}+\sup_{t\in [t_0,T]}||\mu(t)||_{C^{-(1+\alpha)}_c(\Omega)}\leq CM,$$
where $C$ is the constant from (\ref{5.1}) and where
$$M=||v_T||_{2+\alpha}+||\mu_0||_{-(1+\alpha)}+||h||_{0,\alpha}+||c||_{L^1}.$$
\end{prop}

\begin{proof}
The estimate for the second and the third term follows from (\ref{5.1}) and the lower semicontinuity of the norms with respect to the weak/weak-$\star$ convergence. Thus, it suffices to prove the estimate for the first term. Indeed, note that by (\ref{5.1})
$$||v^{\epsilon}||_{1,2+\alpha}\leq CM\text{ in } \Omega_{\epsilon}\; ,$$ 
where the constant $C$ does not depend on $\epsilon$. In consequence, for every $K$ compact subset of $\Omega$, by the convergence of $v^{\epsilon}$ to $v$ we established earlier, we get
$$||v(t,\cdot )||_{C^{2+\alpha}(K)}\leq CM\; .$$
Hence,
$$\sup_{t\in [t_0,T]}||v(t,\cdot)||_{2+\alpha}\leq CM\; ,$$
which is what we wanted.
\end{proof}

\vspace{6mm}

\noindent
We are, now, ready to present the proof of the existence of a solution to (\ref{GL}), when $m_0\in\mathcal{P}(\Omega)$. Before we state the Theorem, we would need the following proposition about the Fokker-Planck equation, which will be needed for the proof of the uniqueness.
\vspace{3mm}

\begin{prop} \label{FPf}
Given the assumptions in Definition 5.1, there exists a unique solution to the Fokker-Planck equation (\ref{fFP}). This solution satisfies
$$\sup_{t\in[t_0,T]}||\mu(t)||_{C^{-(1+\alpha)}_c(\Omega)}+||\mu||_{L^p}\leq C \left( ||\mu_0||_{-(1+\alpha)}+||f||_{L^1(W^{-1,\infty})}\right).$$
\end{prop}

\begin{proof}
\textit{Existence:} Let $\mu^{\epsilon}$ be the solution to the problem
$$\begin{cases}
\partial_t \mu^{\epsilon}-\text{div}(a(x)D\mu^{\epsilon})-\text{div}(\mu^{\epsilon} b)=f ,\text{ in }(t_0,T)\times\Omega_{\epsilon},\\
\mu^{\epsilon}(t_0)=\mu_0|_{C^{1+\alpha}(\Omega_{\epsilon})},\\
(a(x)D\mu^{\epsilon}+\mu^{\epsilon} b)\cdot \nu(x)|_{\partial\Omega_{\epsilon}}=0.
\end{cases}$$
By Proposition 5.3 of [2], we know that $\mu^{\epsilon}$ satisfies the estimate.
$$\sup_{t\in[t_0,T]}||\mu^{\epsilon}(t)||_{C^{-(1+\alpha)}_c(\Omega_{\epsilon})}+||\mu^{\epsilon}||_{L^p}\leq C\left( ||\mu_0||_{-(1+\alpha)}+||f||_{L^1(W^{-1,\infty})}\right).$$
The existence of a solution to (\ref{fFP}) follows by an argument similar as in Proposition \ref{mu} and Step 2 of Proposition \ref{L}.
\vspace{2mm}

\noindent
\textit{Uniqueness:} Follows directly from the estimate.
\vspace{2mm}

\noindent
\textit{Estimate:} The estimate follows from an argument similar to the one in Proposition 5.3 of [2]. Indeed, in the definition, choosing $\psi=0$ we get
$$\langle \mu(t),\xi\rangle =\langle \mu_0,\phi(0,\cdot)\rangle +\int_0^t\langle f(s),\phi(s,\cdot)\rangle ds,$$
hence, taking supremum over $\xi$,
$$||\mu(t)||_{C^{-(1+\alpha)}_c(\Omega)}\leq C\left( ||\mu_0||_{-(1+\alpha)}+\int_0^t||f(s)||_{W^{1,\infty}}ds\right).$$
The first part of the estimate follows. For the second part we choose $\xi=0$. Then,
$$\int_0^T\int_{\Omega}\mu\psi\;dxds=\langle \mu_0,\phi(0,\cdot)\rangle +\int_0^T\langle f(s),\psi(s,\cdot)\rangle \;ds.$$
In consequence, as in Proposition 5.3 of [2], invoking Theorem IV. 9.1, we get $||\phi||_{\frac{1+\alpha}{2},1+\alpha}\leq C||\psi||_{L^{p'}}$ and hence
$$\int_0^T\int_{\Omega}\mu\psi\;dxds\leq C||\psi||_{L^{p'}}\left( ||\mu_0||_{-(1+\alpha)}+\int_0^t||f(s)||_{W^{1,\infty}}ds\right).$$
Thus, since $C^{0,\alpha}$ is dense in $L^{p'}$, the $L^p$ estimate follows.
\end{proof}

\vspace{4mm}

\begin{thm}\label{generalGL}
Assume Hypotheses 2.3. Then, there exists a unique solution \\$(v,\mu)\in C([t_0,T];C^{2+\alpha}(\Omega))\times (C([t_0,T];C^{-(1+\alpha)}_c(\Omega))\cap L^p([t_0,T]\times\Omega))$ of (\ref{GL}) in the sense of Definition 5.2. Furthermore, $v$ is Lipschitz with respect to the $t$ variable and $(v,\mu)$ satisfies the estimate
\begin{equation} \label{ESTIMATE}
    \sup_{t\in [t_0,T]}||v(t,\cdot)||_{2+\alpha}+||\mu||_{L^p}+\sup_{t\in [t_0,T]}||\mu(t)||_{C^{-(1+\alpha)}_c(\Omega)}\leq CM,
\end{equation}
where $C$ is the constant from (\ref{5.1}) and where
$$M=||v_T||_{2+\alpha}+||\mu_0||_{-(1+\alpha)}+||h||_{0,\alpha}+||c||_{L^1}.$$
\end{thm}

\begin{proof} Without loss of generality we assume that $t_0=0$.
\vspace{2mm}

\noindent
\textit{Existence:} By density, there exists a sequence $(m_{0,n})_{n\in\mathbb{N}}\in L^1(\Omega)$ such that\\ $\bm{d_1}(m_{0,n},m_0)\rightarrow 0$. For every $n\in\mathbb{N}$, we consider $(u_n,m_n)$ the solution of \textbf{(MFG)} associated with $(0,m_{0,n})$ given by Theorem 5.1 of [1] and $(v_n,\mu_n)$ the solution of (\ref{GL}), that corresponds to $(u_n,m_n)$, given by Proposition \ref{L}. By Remark \ref{v} \textbf{(ii)}, we know that $(v_n)_{n\in\mathbb{N}}$ is uniformly Lipschitz with respect to the $t$ variable. Also, note that Proposition \ref{es3} implies
$$\sup_{t\in [0,T]}||v_n(t,\cdot)||_{2+\alpha}+||\mu_n||_{L^p}+\sup_{t\in [0,T]}||\mu_n(t)||_{C_c^{-(1+\alpha)}(\Omega)}\leq CM.$$
As a result, arguing as in Proposition 5.3, there exist $v\in C([0,T];C^2(\Omega))$ and $\mu\in C([0,T];C_c^{-(1+\alpha)}(\Omega))\cap L^p(Q_T)$ such that 
\begin{align*}
& v_n(t,\cdot)\rightarrow v(t,\cdot)\text{ uniformly in }C^2(\Omega),\\
& \mu_n\rightharpoonup \mu\text{ weakly in }L^p\text{ and }\\
& \mu_n(t)\stackrel{\star}{\rightarrow}\mu(t) \text{ weak-}\star\text { in }C^{-(1+\alpha)}_c(\Omega), \text{ for any }t\in [0,T],
\end{align*}
up to a subsequence. It, also, follows that $v$ is also Lipschitz with respect to the $t$ variable. We will show that $(v,\mu)$ satisfies the two equations and the boundary conditions of (\ref{GL}) separately. 
\vspace{1mm}

\noindent
The boundary conditions follow easily from the properties of $G$ (Hypothesis 4), because, as we proved in Theorem \ref{MFG sol}, $m_n\rightarrow m$ in $C([0,T];\mathcal{P}(\Omega))$.
\vspace{2mm}

\noindent
Let $\phi\in C^{\infty}_c((0,T]\times \Omega)$. Then, we have
\begin{equation}\label{5.11}
    \begin{split}
\int_0^T\int_{\Omega}v_n(\partial_t \phi)\;dxdt-
\int_0^T\int_{\Omega}\text{tr}(a(x)D^2v_n)\phi\; dxdt+\int_0^T\int_{\Omega}H_p(x,Du_n)\cdot Dv_n\phi \;dxdt=\\ \int_{0}^T\int_{\Omega}h(t,x)\phi(t,x)\; dxdt+ \int_{\Omega}\int_0^T\int_{\Omega}\frac{\delta F}{\delta m}(x,m_n(t),y)\phi(x,t)\mu_n(t,y)\;dydtdx\;.
\end{split}
\end{equation}
Passing to the limit in (\ref{5.11}), since $m_n\rightarrow m$ in $L^p$ (Proposition \ref{5.6.[2]}) and $u_n(t,\cdot)\rightarrow u(t,\cdot)$ in $C^2$ (Theorem \ref{MFG sol}), we get
\begin{equation*}
    \begin{split}
\int_0^T\int_{\Omega}v(\partial_t \phi)\;dxdt-
\int_0^T\int_{\Omega}\text{tr}(a(x)D^2v)\phi\; dxdt+\int_0^T\int_{\Omega}H_p(x,Du)\cdot Dv\phi \;dxdt=\\ \int_{0}^T\int_{\Omega}h(t,x)\phi(t,x)\; dxdt+ \int_{\Omega}\int_0^T\int_{\Omega}\frac{\delta F}{\delta m}(x,m(t),y)\phi(x,t)\mu(t,y)\;dydtdx\;,
\end{split}
\end{equation*}
hence we conclude that $(v,\mu)$ satisfies the first equation of (\ref{GL}). 
\vspace{2mm}

\noindent
To prove that the second equation is satisfied, let $\phi$ be a solution to (\ref{eqtest}) and $\phi_n$ be a solution to
$$\begin{cases}
-\partial_t\phi_n-\text{tr}(a(x)D^2\phi_n)+H_p(x,Du_n)\cdot D\phi_n=
-\partial_t\phi-\text{tr}(a(x)D^2\phi)+H_p(x,Du)\cdot D\phi,\\
\hfill \hfill \hfill\hfill\hfill\hfill \text{ in }(0,t)\times\Omega,\\
\phi_n(t)=\xi,\text{ in }\Omega .\\
\end{cases}$$
Notice that, (\ref{result}) and Arzela-Ascoli give $\phi_n(0,\cdot)\rightarrow \phi(0,\cdot)$ in $C^{1+\alpha}$ and $\phi_n\rightarrow \phi$ in $C^{0,1}$, up to a subsequence. By Proposition \ref{L}, $(v_n,\mu_n)$ satisfies 
\begin{multline}
\int_0^t\int_{\Omega}\mu_n(-\partial_t\phi-\text{tr}(a(x)D^2\phi)+H_p(x,Du)\cdot D\phi)dxds-\langle \mu_0,\phi_n(0)\rangle=\\
-\int_0^t\int_{\Omega}c\cdot D\phi_n\;dxds-\int_0^t\int_{\Omega}m_nH_{pp}(x,Du_n)Dv_n\cdot D\phi_n dxds-\langle \mu_n(t),\xi\rangle.
\end{multline}
After passing to the limit and having all the convergences in mind, it follows that
\begin{multline*}
\int_0^t\int_{\Omega}\mu(-\partial_t\phi-\text{tr}(a(x)D^2\phi)+H_p(x,Du)\cdot D\phi)dxds-\langle \mu_0,\phi(0)\rangle=\\
-\int_0^t\int_{\Omega}c\cdot D\phi\;dxds-\int_0^t\int_{\Omega}mH_{pp}(x,Du)Dv\cdot D\phi\; dxds-\langle \mu(t),\xi\rangle.
\end{multline*}
Thus, $(v,\mu)$ satisfies the second equation of (\ref{GL}).
\vspace{2mm}

\noindent
\textit{Uniqueness:} Uniqueness follows by observing that if $(v_1,\mu_1)$ and $(v_2,\mu_2)$ are two solutions of (\ref{GL}), then $(v_1-v_2,\mu_1-\mu_2)$ satisfies (\ref{GL}) for $h=c=\mu_0=v_T=0$. Hence, (\ref{ESTIMATE}) implies that $\mu_1=\mu_2$ and $v_1=v_2$. 
\vspace{2mm}

\noindent
\textit{Estimate:} The proof of the estimate is an adaptation of the proof of the estimate in Proposition 5.8 of [2], so we provide a sketch. 

\noindent
We know that $v$ can be uniquely obtained as the limit of solutions to the problems
$$\begin{cases}
-\partial_t v^{\epsilon}-\text{tr}(a(x)D^2v^{\epsilon})+H_p(x,Du)\cdot Dv^{\epsilon}=\frac{\delta F}{\delta m}(x,m(t))(\mu (t))+h(t,x),\text{ in }(0,T)\times\Omega_{\epsilon},\\
v^{\epsilon}(T,x)=\frac{\delta G^{\epsilon}}{\delta m}(x,m(T))(\mu(T))+v_T(x)-g_{\epsilon}(x),\text{ in }\Omega_{\epsilon},\\
a(x)Dv^{\epsilon}\cdot \nu(x)|_{\partial\Omega_{\epsilon}}=0,
\end{cases}$$
as $\epsilon\rightarrow 0$ (Proposition 3.4 of [1]). Using Theorem 5.1.21 from [6], we find

\begin{align*}
||v^{\epsilon}||_{1,2+\alpha}&\leq C\left(||v_T-g_{\epsilon}||_{2+\alpha}+\sup_{t\in [0,T]}||\mu(t)||_{C^{-(2+\alpha)}_c}+||h||_{0,\alpha}\right)\\
&\leq C\left( M+\sup_{t\in [0,T]}||\mu(t)||_{C^{-(1+\alpha)}_c}\right),
\end{align*}
so $v$ satisfies
$$\sup_{t\in[0,T]}||v(t,\cdot)||_{2+\alpha}\leq C\left(M+\sup_{t\in [0,T]}||\mu(t)||_{C^{-(1+\alpha)}_c}\right).$$
In addition, by writing Definition 5.1 for
$$\psi=\frac{\delta F}{\delta m}(x,m(t))(\mu(t))+h(t,x)\in C^{0,\alpha},\;\;\xi=\frac{\delta G}{\delta m}(x,m(T))(\mu(T))+v_T(x)\in C^{1+\alpha}_c$$
and using the monotonicity of $F$, $G$, we get
$$\int_0^T\int_{\Omega}m\langle H_{pp}(x,Du)Dv,Dv\rangle\;dxdt\leq M\left(\sup_{t\in [0,T]}||\mu(t)||_{C^{-(1+\alpha)}_c(\Omega)}+||\mu||_{L^p}+\sup_{t\in[0,T]}||v(t,\cdot)||_{2+\alpha}\right).$$
Furthermore, as in Proposition 5.8 of [2],
$$||mH_{pp}(x,Du)Dv||_{L^1}\leq M^{1/2}\left( \sup_{t\in [0,T]}||\mu(t)||_{C^{-(1+\alpha)}_c(\Omega)}^{1/2}+||\mu||_{L^p}^{1/2}+\sup_{t\in [0,T]}||v(t,\cdot)||_{2+\alpha}^{1/2} \right).$$
Combining these with the estimate from Proposition \ref{FPf} and the Arithmetic-Geometric mean inequality, (\ref{ESTIMATE}) follows.
\end{proof}

\subsection{Derivative with respect to the measure}

\vspace{5mm}

\noindent
For $y\in \Omega$, we define
\begin{equation}\label{Kdef}
K(t_0,x,m_0,y):=v(t_0,x; \delta_y),
\end{equation}
where $v(t_0,x;\delta_y)$ comes from the solution of \textbf{(L)} with $\mu_0=\delta_y$.
\vspace{4mm}

\noindent
Before we proceed with discovering the properties of $K$, we start with a useful lemma.
\vspace{2mm}

\begin{lem} \label{Stability} (Stability of solutions)

\noindent
Assume Hypotheses 2.3. Let $(\mu_{0,n})_{n\in\mathbb{N}}$ be a sequence in $C^{-(1+\alpha)}(\Omega)$ and $\mu_0\in C^{-(1+\alpha)}(\Omega)$ such that $\mu_{0,n}\rightarrow \mu_0$. Consider $v_n,v$ to be the functions we get from the solution of \textbf{(L)} starting at $\mu_{0,n}$ and $\mu_0$, respectively. Then, 
$$v_n(t_0,\cdot)\rightarrow v(t_0,\cdot)\text{ in }C^{2+\alpha},$$
for any $t_0\in [0,T]$.
\end{lem}

\begin{proof}
By (\ref{ESTIMATE}) and the linearity, we deduce
\begin{equation} \label{stab}
\sup_{t\in [0,T]} ||v_n(t,\cdot)-v(t,\cdot)||_{2+\alpha}\leq C||\mu_{0,n}-\mu_0||_{-(1+\alpha)}
\end{equation}
The result follows.
\end{proof}
\vspace{4mm}

\noindent
Now we are ready to prove the properties of $K$.
\vspace{3mm}

\begin{prop} \label{representation}
Assume Hypotheses 2.3. Then, the function $K: [0,T]\times \Omega\times \mathcal{P}(\Omega)\times \Omega$ satisfies \begin{equation}\label{repr}
v(t_0,x)=\langle \mu_0, K(t_0,x,m_0,\cdot)\rangle_{1+\alpha}
\end{equation}
for all $(t_0,x,\mu_0)\in [0,T]\times\Omega\times C^{-(1+\alpha)}$. Moreover, $K(t_0,\cdot,m_0,\cdot)\in C^{2+\alpha,1+\alpha}$ with
$$\sup_{(t,m)\in[t_0,T]\times\mathcal{P}(\Omega)}||K(t,\cdot,m,\cdot)||_{2+\alpha,1+\alpha}\leq C.$$
\end{prop}

\begin{proof}
The fact that $K$ is $C^{2+\alpha}$ with respect to $x$ follows from (\ref{Kdef}). So, we only need to prove the $C^{1+\alpha}$ character of $K$ with respect to $y$. By linearity and the uniqueness of the solution of \textbf{(L)}, we have
$$\frac{K(t_0,x,m_0,y+he_j)-K(t_0,x,m_o,y)}{h}=v(t_0,x;\Delta_{h,j}\delta_y)\;,$$
where $\Delta_{h,j}\delta_y=\frac{1}{h}(\delta_{y+he_j}-\delta_y)$. Sending $h\rightarrow 0$ and using (\ref{stab}), we deduce that
$$\frac{\partial K}{\partial y_j}(t_0,x,m_0,y)=v(t_0,x;-\partial_{y_j}\delta_y)\; ,$$
which is $C^{2+\alpha}$ with respect to $x$ and, due to (\ref{stab}), and $\alpha-$Hölder continuous with respect to $y$. The estimate follows from Proposition \ref{es3}, (\ref{stab}) and the fact that $\partial_{y_j}\delta_y$ is bounded in $C^{-(1+\alpha)}$ for all $j$.
\vspace{1mm}

\noindent
Finally, note that, (\ref{Kdef}) implies that the representation formula (\ref{repr}) holds when $\mu_0=\delta_y$ for some $y\in\Omega$. From the density of the subspace generated by the Dirac delta functions, in the weak-$\star$ topology, it follows that we can extend (\ref{Kdef}) to (\ref{repr}). The proof is complete. 
\end{proof}

\vspace{4mm}

\begin{prop} \label{diff}
Suppose that Hypotheses 2.3 are satisfied. Then, $U$ is differentiable with respect to the measure variable and, more specifically, $$K(t_0,x,m_0,y)=\frac{\delta U}{\delta m}(t_0,x,m_0,y)$$
for all $(t_0,x,m_0,y)\in [0,T]\times \Omega\times \mathcal{P}(\Omega)\times \Omega$.
\end{prop}

\begin{proof}
\noindent
Consider $m_{01},m_{02}\in\mathcal{P}(\Omega)$ and let $(u_1,m_1)$ and $(u_2,m_2)$ be the solutions of \textbf{(MFG)} associated with the starting initial conditions $(0,m_{01})$ and $(0,m_{02})$, respectively.  Let also $(v,\mu)$ be the solution of \textbf{(L)} related to $(u_2,m_2)$ with initial condition $(0,m_{01}-m_{02})$.
Then, $(u_1-u_2-v,m_1-m_2-\mu)$ solves a system of the same form as (\ref{GL}) and hence (\ref{ESTIMATE}) and the same calculations as in Theorem 5.10 in [2] yield
$$\sup_{t\in [0,T]}||u_1(t,\cdot)-u_2(t,\cdot)-v(t,\cdot)||_{2+α}\leq C\bm{d}_1(m_{01},m_{02})^2\text{ in }\Omega.$$

\noindent
In particular, we can write
$$
||u_1(t_0,\cdot)-u_2(t_0,\cdot)-v(t_0,\cdot)||_{2+\alpha}\leq C\bm{d}_1(m_{01},m_{02})^2\; $$
\vspace{1mm}

\noindent
Now, by Proposition \ref{representation} we deduce
\begin{equation} \label{2}
\left|\left|U(t_0,\cdot,m_{01})-U(t_0,\cdot,m_{02})-\int_{\Omega}K(t_0,\cdot,m_{02},y)(m_{01}-m_{02})dy\right|\right|_{2+\alpha}\leq C\bm{d}_1(m_{01},m_{02})^2\; .
\end{equation}

\noindent
Let $m,\widetilde{m}\in\mathcal{P}(\Omega)$. For $s\in [0,1]$, choose $m_{02}=m$ and $m_{01}=(1-s)m+s\widetilde{m}$. Then, (5.19) gives
\begin{multline*}
\left|\left|U(t_0,\cdot,(1-s)m+s\widetilde{m})-U(t_0,\cdot,m)-s\int_{\Omega}K(t_0,\cdot,m,y)(\widetilde{m}-m)dy\right|\right|_{2+\alpha}\\
\leq C\bm{d}_1((1-s)m+s\widetilde{m},m)^2
\end{multline*}
or
\begin{multline*}
\left|\left|\frac{U(t_0,\cdot,(1-s)m+s\widetilde{m})-U(t_0,\cdot,m)}{s}-\int_{\Omega}K(t_0,\cdot,m,y)(\widetilde{m}-m)dy\right|\right|_{2+\alpha}\\
\leq sC\bm{d}_1(\widetilde{m},m)^2.
\end{multline*}

\noindent
By letting $s\rightarrow 0$, we deduce
$$\left|\left|\int_{\Omega}\frac{δU}{δm}(t_0,\cdot,m,y)(\widetilde{m}-m)dy-\int_{\Omega}K(t_0,\cdot,m,y)(\widetilde{m}-m)dy\right|\right|_{2+\alpha}=0.$$
Thus, 
$$\frac{δU}{δm}(t_0,\cdot,m,\cdot)=K(t_0,\cdot,m,\cdot) \text{ in  } \Omega\times \Omega$$
The result follows. 
\end{proof}

\vspace{4mm}

\noindent
Combining Propositions \ref{representation} and \ref{diff}, we deduce that $\frac{\delta U}{\delta m}(t_0,x,m_0,y)$ is differentiable with respect to $y$ and hence $D_mU(t_0,x,m_0,y)$ is defined. However, it is also true that $K=\frac{\delta U}{\delta m}$ has a second derivative with respect to $y$. We will prove this result in Proposition \ref{diff2}, but before we state it, we are going to need the following lemma.
\vspace{7mm}

\begin{lem} \label{est}
Suppose that Hypotheses 2.3 are satisfied and $\mu_0\in C^{-(1+\alpha)}(\Omega)$. Let $(v,\mu)$ be the solution to \textbf{(L)}. Then, there exists a constant $C$, independent of $\mu_0$, such that
$$\sup_{t\in [0,T]}||v(t,\cdot )||_{2+\alpha}\leq C||\mu_0||_{-(2+\alpha)}$$
\end{lem}

\begin{proof}
Let $t\in [0,T]$. 
\vspace{1mm}

\noindent
\textit{Step 1:} We will start with the case where $m_0\in L^1(\Omega)\cap\mathcal{P}(\Omega)$. Note that by Proposition 5.11 of [2], we have
$$||v^{\epsilon}||_{1,2+\alpha}\leq C||\mu_0||_{-(2+\alpha)}\;,\text{  in }\Omega_{\epsilon},$$
where the constant $C$ does not depend on $\epsilon$ and where $v^{\epsilon}$ comes from the solution of (\ref{epsilonGL}). In particular, by the convergence of $v^{\epsilon}$ to $v$, this implies that for every $K$ compact subset of $\Omega$
$$||v(t,\cdot )||_{C^{2+\alpha}(K)}\leq C||\mu_0||_{-(2+\alpha)}.$$
Hence,
$$\sup_{t\in [0,T]}||v(t,\cdot)||_{2+\alpha}\leq C||\mu_0||_{-(2+\alpha)},$$
which is what we wanted.
\vspace{2mm}

\noindent
\textit{Step 2:} Consider $m_0\in \mathcal{P}(\Omega)$. Then, by density, there exists a sequence $(m_{0,n})_{n\in\mathbb{N}}\in L^1(\Omega)$ such that $\bm{d_1}(m_{0,n},m_0)\rightarrow 0$. Now, the result in Step 1  implies
$$\sup_{t\in [0,T]}||v_n(t,\cdot)||_{2+\alpha}\leq C||\mu_0||_{-(2+\alpha)},$$
where $v_n$ comes from the solution to (\ref{GL}) related to $(0,m_{0,n})$ and $\mu_0$. The result follows from the convergence of $v_n$ to $v$ established in Theorem \ref{generalGL}.\end{proof}

\vspace{5mm}

\begin{prop} \label{diff2}
Assume Hypotheses 2.3. Then, the function $\frac{\delta U}{\delta m}$ is twice differentiable with respect to $y$, together with its first and second derivatives with respect to $x$. Moreover, the following estimate holds:
$$\left|\left|\frac{\delta U}{\delta m}(t,\cdot,m,\cdot)\right|\right|_{2+\alpha,2+\alpha}\leq C\;,$$
for some constant $C>0$.
\end{prop}

\begin{proof}
The proof goes along the lines of Corollary 5.12 of [2]. Let $K$ be a compact subset of $\Omega$ with nonempty interior. By Proposition \ref{representation}, we already know that 
$$\partial_{y_i}\frac{\delta U}{\delta m}(t_0,x,m_0,y)=v(t_0,x;-\partial_{y_i}\delta_y).$$
We consider the ratio
$$R_{i,j}^h(x,y)=\frac{v(t_0,x;-\partial_{y_i}\delta_{y+he_j})-v(t_0,x;-\partial_{y_i}\delta_y)}{h}.$$
We will show that it is Cauchy with respect to $h$ and hence it converges. Indeed, let $h,k>0$ and $0\leq |l|\leq 2$. By uniqueness of a solution and the linearity of \textbf{(L)}, we have
\begin{equation} \label{a}
D_x^lR_{i,j}^h(x,y)-D_x^lR_{i,j}^k(x,y)= D_x^lv(t_0,x;\Delta_h^j(-\partial_{y_i}\delta_y)-\Delta_k^j(-\partial_{y_i}\delta_y))\; ,
\end{equation}
where $\Delta_h^j(-\partial_{y_i}\delta_y)=-\frac{1}{h}(\partial_{y_i}\delta_{y+he_j}-\partial_{y_i}\delta_y)$. Now, let $x\in K$. It follows from Lemma \ref{est} and the mean value theorem that
\begin{align*}
|D_x^lR_{i,j}^h(x,y)-D_x^lR_{i,j}^k(x,y)|& \leq C||\Delta_h^j(-\partial_{y_i}\delta_y)-\Delta_k^j(-\partial_{y_i}\delta_y)||_{-(2+\alpha)}\quad\quad\quad\quad\quad\quad\quad\\
 = C \sup_{||\phi||_{2+\alpha}\leq 1}& \left(\frac{\partial_{y_i}\phi(y+he_j)-\partial_{y_i}\phi(y)}{h}-\frac{\partial_{y_i}\phi(y+ke_j)-\partial_{y_i}\phi(y)}{k}\right)\\
& =C\sup_{||\phi||_{2+\alpha}\leq 1}\left(\partial^2_{y_iy_j}\phi(y_{\phi,h})-\partial^2_{y_iy_j}\phi(y_{\phi,k})\right)\\
& \leq C\sup_{||\phi||_{2+\alpha}\leq 1}|y_{\phi,h}-y_{\phi,k}|^{\alpha} \\
& \leq C(|h|^{\alpha}+|k|^{\alpha}),
\end{align*}
hence the ratio is Cauchy. We deduce that $\partial_{y_i}\frac{\delta U}{\delta m}(t_0,x,m_0,y)$ is differentiable with respect to $y$, as claimed.
\vspace{2mm}

\noindent
To prove the estimate we argue as follows. Combining (\ref{a}) and Lemma \ref{est} we discover 
\begin{equation}
  ||R_{i,j}^h(x,y)-R_{i,j}^k(x,y)||_{2+\alpha}  \leq C||\Delta_h^j(-\partial_{y_i}\delta_y)-\Delta_k^j(-\partial_{y_i}\delta_{y'})||_{-(2+\alpha)}.
\end{equation}
Now, we know that
$$\Delta_h^j(-\partial_{y_i}\delta_y)-\Delta_k^j(-\partial_{y_i}\delta_{y'})\rightarrow \partial_{y_j}\partial_{y_i}\delta_y-\partial_{y_j}\partial_{y_i}\delta_{y'}\quad\text{ in }C^{-(2+\alpha)}$$
and $D_x^lR_{i,j}^h(x,y)\rightarrow \partial^2_{y_iy_j}D_x^l\frac{\delta U}{\delta m}(t_0,x,m,y)$ for all $|l|\leq 2$, as $h\rightarrow 0$, thus
$$\left|\left|\partial^2_{y_iy_j}\frac{\delta U}{\delta m}(t_0,\cdot,m,y)-\partial^2_{y_iy_j}\frac{\delta U}{\delta m}(t_0,\cdot,m,y')\right|\right|_{2+\alpha}\leq C||\partial_{y_j}\partial_{y_i}\delta_y-\partial_{y_j}\partial_{y_i}\delta_{y'}||_{-(2+\alpha)}\leq C|y-y'|^{\alpha}.$$
The proof is complete. \end{proof}

\vspace{5mm}

\subsection{Further estimates}
We end this section by providing two extra regularity results that are used in the proof of the convergence problem (section 7). In particular, we will prove that $\frac{\delta U}{\delta m}$ is Lipschitz with respect to the measure, as well as a more general estimate, which is similar to (\ref{2}).

\begin{thm}\label{Lip dU}
Suppose Hypotheses 2.3 are satisfied. Then, $\frac{\delta U}{\delta m}$ is Lipschitz continuous with respect to the measure variable. More specifically, the following estimate holds
$$\sup_{t\in [t_0,T]}\sup_{m_1\neq m_2}\bm{d_1}(m_1,m_2)^{-1}\left|\left|\frac{\delta U}{\delta m}(t,\cdot,m_1,\cdot)-\frac{\delta U}{\delta m}(t,\cdot,m_2,\cdot)\right|\right|_{2+\alpha,1+\alpha}\leq C$$
\end{thm}

\begin{proof}
Let $m_{01},m_{02}\in\mathcal{P}(\Omega)$ and $(u_1,m_1)$, $(u_2,m_2)$ be the solutions to \textbf{(MFG)} starting at $(t_0,m_{01})$, $(t_0,m_{02})$, respectively. Also, we consider $\mu_0\in C^{-(1+\alpha)}$ and, for $i\in \{1,2\}$, let $(v_i,\mu_i)$ be the solution to \textbf{(L)} associated with $(u_i,m_i)$ starting at $\mu_0$. Then, $(v,\mu):= (v_1-v_2,\mu_1-\mu_2)$ is the solution to (\ref{GL}) with 
\begin{align*}
&\mu_0 =0,\\
&h(t,x)=h_1(t,x)+h_2(t,x),\\
&h_1(t,x)=\frac{\delta F}{\delta m}(x,m_1(t))(\mu_2(t))-\frac{\delta F}{\delta m}(x,m_2(t))(\mu_2(t)),\\
&h_2(t,x)=(H_p(x,Du_1)-H_p(x,Du_2))\cdot v_2(t,x),\\
&c(t,x)=\mu_2(t)(H_p(x,Du_1)-H_p(x,Du_2))+[(m_1H_{pp}(x,Du_1)-m_2H_{pp}(x,Du_2)](t,x),\\
&v_T(x)=\frac{\delta G}{\delta m}(x,m_1(T))(\mu_2(T))-\frac{\delta G}{\delta m}(x,m_2(T))(\mu_2(T)),
\end{align*}
so by Theorem \ref{generalGL}
$$\sup_{t\in [t_0,T]}||v(t,\cdot)||_{2+\alpha} \leq C(||v_T||_{2+\alpha}+||h||_{0,\alpha}+||c||_{L^1})\;,$$
where $C$ is the constant from (\ref{5.1}). Now, using our Hypotheses on $F$, $G$, $H$ and Lemma
\ref{stability estimates}, we can estimate the terms on the right-hand side as in Theorem 3.3 in [16] to get
$$\sup_{t\in[t_0,T]}||v(t,\cdot)||_{2+\alpha}\leq C\bm{d_1}(m_{01},m_{02})||\mu_0||_{-(1+\alpha)}\;.$$
Since
$$v(t_0,x)=\int_{\Omega}\left(\frac{\delta U}{\delta m}(t_0,x,m_{01},y)-\frac{\delta U}{\delta m}(t_0,x,m_{02},y)\right)\mu_0(dy),$$
the result follows from Proposition \ref{representation}.
\end{proof}
\vspace{4mm}

\begin{lem} \label{general estimate}
If $m\in\mathcal{P}(\Omega)$ and $\phi\in L^2(m,\mathbb{R}^d)$ is a bounded vector field such that $\text{Im}(id+\phi)\subseteq \Omega$, then:
$$\left|\left|U(t,\cdot, (id+\phi)\#m)-U(t,\cdot,m)-\int_{\Omega}D_mU(t,\cdot,m,y)\cdot \phi(y)dm(y)\right|\right|_{1+\alpha}\leq C||\phi||^2_{L^2(m)}\;.$$
\end{lem}

\begin{proof}
By Theorem \ref{Lip dU}, the proof is an adaptation of Propositions A.2.1 and A.2.2 of [5].
\end{proof}

\vspace{9mm}

\section{Solvability of the Master Equation}

\noindent
The results from the previous sections indicate that $U$ from (\ref{ME}) has the regularity to be the unique solution to the Master Equation.
\vspace{2mm}

\begin{defn}
 Let $0<\alpha<1$. A map $U\in\mathcal{C}^{1+\frac{\alpha}{2},2+\alpha,1}([0,T]\times \Omega\times\mathcal{P}(\Omega))$ is a \textit{classical solution} of the \textbf{ME} if:
 \vspace{1mm}
 
 \noindent
 (i) $\frac{δU}{δm}(t,x,m,y)$ is Lipschitz continuous in all the arguments and\\           
 \noindent
 \textcolor{white}{.....}$\frac{δU}{δm}(t,\cdot ,m,\cdot )\in C^{2+\alpha,2+\alpha}(\Omega\times\Omega)$ for every $(t,m)\in [0,T]\times\mathcal{P}(\Omega)$.
 \vspace{2mm}
 
 \noindent
 (ii) $U$ satisfies \textbf{ME} in the classical sense.
\end{defn}
\vspace{5mm}

\noindent
Before we verify that $U$ is classical solution to the \textbf{ME}, we will prove a lemma, which will be useful in the calculations.
\vspace{3mm}

\begin{lem} (Integration by parts)\label{int by parts}\\
Let $t_0\in [0,T)$, $h\in (0,T-t_0)$, $m_0\in L^1(\Omega)$ and $m$ be a solution to (\ref{FP}), where (\ref{PR}) is assumed near the boundary. Then, for any $\varphi \in C([0,T]; L^1(\Omega))\cap L^{\infty}([0,T]\times\Omega)$ such that 
$$\begin{cases}
-\partial_t\varphi -\text{tr}(a(x)D^2\varphi)+\alpha\cdot D\varphi\in L^{\infty}(Q_T)\; ,\\
\varphi(T)=\psi\in L^{\infty}(\Omega)
\end{cases}$$
in the sense of distributions, the following formula holds:
\begin{equation} \label{int1}
\int_{t_0}^{t_0+h}\int_{\Omega}m(-\partial_t\varphi -\text{tr}(a(x)D^2\varphi)+\alpha\cdot D\varphi)\;dxdt=\int_{\Omega}m_0\varphi (0)\;dx-\int_{\Omega}m(t_0+h)\varphi (t_0+h)\;dx\; .
\end{equation}
Furthermore, if we assume that $m\in C^{1+\frac{\alpha}{2},2+\alpha}$ and $\phi$ is Lipschitz continuous with respect to $t$, then 
\begin{equation}\label{int2}
    \int_{\Omega}m_0\left(-\text{tr}(a(x)D^2\varphi(t_0))+\alpha\cdot D\varphi (t_0)\right)\; dx=\int_{\Omega}\partial_t m(t_0)\varphi (t_0)\;dx\; .
\end{equation}
\end{lem}
\vspace{2mm}

\begin{proof}
By Defintion 4.1, we may write the equalities
\begin{equation*} 
\int_{t_0}^{T}\int_{\Omega}m(-\partial_t\varphi -\text{tr}(a(x)D^2\varphi)+\alpha\cdot D\varphi)\;dxdt=\int_{\Omega}m_0\varphi (t_0)\;dx-\int_{\Omega}m(T)\psi\;dx 
\end{equation*}
and
\begin{equation*} 
\int_{t_0+h}^{T}\int_{\Omega}m(-\partial_t\varphi -\text{tr}(a(x)D^2\varphi)+\alpha\cdot D\varphi)\;dxdt=\int_{\Omega}m(t_0+h)\varphi (t_0+h)\;dx-\int_{\Omega}m(T)\psi\;dx \; .
\end{equation*}
(\ref{int1}) follows after subtracting them.
\vspace{3mm}

\noindent
Now, suppose that $m$ is a classical solution of (\ref{FP}). Integrating by parts with respect to $t$ and using (\ref{int1}), we discover
\begin{equation*}
    \int_{t_0}^{t_0+h}\int_{\Omega}(\partial_t m) \varphi \text{ }dydt=\int_{t_0}^{t_0+h}\int_{\Omega}m(-\text{tr}(a(y)D^2\varphi)+\alpha\cdot D\varphi )\; dydt\; .
\end{equation*}
Dividing by $h$ and letting $h\rightarrow 0^+$, yields (\ref{int2}).
\end{proof}

\vspace{7mm}

\noindent
Now, we are ready to prove \textbf{Theorem} \ref{MEsol}.

\begin{proof}
The fact that $U(T,x,m)=G(x,m)$ follows from the definition of $U$. Let $m_0\in \mathcal{P}(\Omega)\cap L^1(\Omega)$ smooth with compact support and $t_0\in [0,T]$. We have
\begin{multline}\label{basic}
\frac{U(t_0+h,x,m_0)-U(t_0,x,m_0)}{h}=\frac{U(t_0+h,x,m_0)-U(t_0+h,x,m(t_0+h))}{h}\\
+\frac{U(t_0+h,x,m(t_0+h))-U(t_0,x,m_0)}{h}\;.
\end{multline}

\noindent
For the second term, we know
$$\lim_{h\rightarrow 0}\frac{U(t_0+h,m(t_0+h))-U(t_0,m_0)}{h}=\partial_t u(t_0,x).$$

\noindent
On the other hand, letting $m_{s,h}=(1-s)m_0+sm(t_0+h)$,

\begin{align*}
U(t_0+h,x,m(t_0+h))-& U(t_0+h,x,m_0)   \\
& =\int_0^1\int_{\Omega}\frac{\delta U}{\delta m}(t_0+h,x,m_{s,h},y)(m(t_0+h)-m_0)dyds
\\
& = \int_0^1\int_{\Omega}\int_{t_0}^{t_0+h}\frac{\delta U}{\delta m}(t_0+h,x,m_{s,h},y)\partial_t m\text{ }dtdyds \; .
\end{align*}

\noindent
Thus, since $\lim_{h\rightarrow 0}\frac{\delta U}{\delta m}(t_0+h,x,m_{h,s},y)=\frac{\delta U}{\delta m}(t_0,x,m_0,y)$ uniformly, we have
$$\lim_{h\rightarrow 0}\frac{U(t_0+h,x,m(t_0+h))-U(t_0+h,x,m_0)}{h}=\int_{\Omega}\frac{\delta U}{\delta m}(t_0,x,m_0,y)\partial_t m(t_0)\;dy\; .$$
However, since 
$$-\partial_t v-\text{tr}(a(x)D^2v)+H_p(x,Du)\cdot Dv\in L^{\infty}(Q_T),$$
$m\in C^{1+\frac{\alpha}{2},2+\alpha}$ (from Theorem \ref{MFG sol}) and $v$ is Lipschitz with respect to $t$ (from Remark 5.4), by virtue of Lemma \ref{int by parts}, we can write
\begin{multline*}
\int_{\Omega}m_0(-\text{div}(a(y)Dv(t_0,x;y))+(H_p(x,Du(t_0,y)+\tilde{b}(y))\cdot Dv(t_0,x;y))\; dy\\=\int_{\Omega} \partial_tm(t_0) v(t_0,x;y)dy\; ,
\end{multline*}
hence
\begin{multline*}
\lim_{h\rightarrow 0}\frac{U(t_0+h,x,m(t_0+h))-U(t_0+h,x,m_0)}{h}\\=
\int_{\Omega}m_0(-\text{tr}(a(y)D^2v(t_0,x;y))+H_p(x,Du(t_0,y))\cdot Dv(t_0,x;y))dy\; .
\end{multline*}

\noindent
From (\ref{basic}), we deduce that
\begin{multline*}
\lim_{h\rightarrow 0}\frac{U(t_0+h,x,m_0)-U(t_0,x,m_0)}{h}\\=\int_{\Omega}m_0(-\text{tr}(a(y)D^2v(t_0,x;y))+H_p(y,Du(t_0,y))\cdot Dv(t_0,x;y))\; dy-\partial_t u(t_0,x)\; .
\end{multline*}
Since $v(t_0,x;y)=\frac{\delta U}{\delta m}(t_0,x,m_0,y)$ (relation (5.16)), we conclude 
\begin{multline*}
\partial_t U(t_0,x,m_0)=\\\int_{\Omega}m_0\left(-\text{tr}(a(y)D_yD_mU(t_0,x,m_0,y))\right)\;dy+
\int_{\Omega}H_p(y,DU(t_0,y,m_0))\cdot D_mU(t_0,x,m_0,y)m_0\;dy\\
-\text{tr}(a(x)D^2_xU(t_0,x,m_0))+H(x,DU(t_0,x,m_0)-F(x,m_0).
\end{multline*}
The result for general $m_0\in \mathcal{P}(\Omega)$ follows from a standard density argument.  \end{proof}
\vspace{5mm}

\noindent
Finally, we will prove that $U$ is the \textbf{unique classical solution} to the \textbf{ME}. The idea is to derive \textbf{(MFG)} from a solution of the \textbf{ME}. To establish that, we would need the following lemma.
\vspace{4mm}

\begin{lem} \label{General FP}
Assume $a$, $H$ satisfy Hypotheses 1, 2 and 7 and let $m_0\in \mathcal{P}(\Omega)$ with density in $C^{2+\alpha}$ with compact support. Also, let $f:[t_0,T]\times\Omega\times \mathcal{P}(\Omega)\rightarrow \mathbb{R}$ be a bounded function which is of class $C^{1+\frac{\alpha}{2},1+\alpha}$ over spacetime and Lipschitz continuous with respect to the measure variable. Then, there exists a solution, in the sense of Definition 4.1, of the equation
$$\begin{cases}
\partial_tm-\text{div}(a(x)Dm)-\text{div}\left(m(H_p(x,f(t,x,m))+\tilde{b}(x))\right)=0\quad,\text{ in }(t_0,T)\times \Omega,\\
m(t_0)=m_0,
\end{cases}$$
where near the boundary we assume $(2.1)$. Furthermore, the solution $m$ is of class $ C^{1+\frac{\alpha}{2},2+\alpha}$.\end{lem}
\vspace{3mm}

\begin{proof}
We will use a fixed point argument. Let 
$$X=\{ m\in C([t_0,T];\mathcal{P}(\Omega)\text{ : }\bm{d_1}(m(t),m(s))\leq L|t-s|^{1/2}\}\;,$$
where $L$ is to be chosen later. Note that $X$ compact and convex. We define the map 
$$T:X\rightarrow X\text{ with } T(\mu)=m\; ,$$
where $m\in C([0,T];L^1(\Omega))$ is the unique weak solution, in the sense of Definition 4.1, to the equation 
\begin{equation} \label{xxxxx}
\begin{cases}
\partial_tm-\text{div}(a(x)Dm)-\text{div}\left(m(H_p(x,f(t,x,\mu))+\tilde{b}(x))\right)=0,\text{ in }(t_0,T)\times\Omega,\\
m(t_0)=m_0,
\end{cases}
\end{equation}
where near the boundary we assume $(2.1)$, as given by Theorem 4.5 of [1]. We separate the proof in three steps.
\vspace{2mm}

\noindent
\textit{Step 1:} We will show that $T$ is well defined. Indeed, consider the problem
\begin{equation} \label{zzzzz}
\begin{cases}
\partial_t m^{\epsilon}-\text{div}(a(x)Dm^{\epsilon})-\text{div}\left(m^{\epsilon}(H_p(x,f(t,x,\mu))+\tilde{b}(x))\right)=0,\text{ in }(t_0,T)\times \Omega_{\epsilon},\\
m^{\epsilon}(t_0)=m_0,\\
a(x)Dm^{\epsilon}+m^{\epsilon}\left(H_p(x,f(t,x,\mu))+\tilde{b}(x)\right)\cdot\nu(x)|_{\partial\Omega_{\epsilon}}=0.
\end{cases}
\end{equation}
By Proposition 4.3 of [1], since $H_p$ is Lipschitz and $f$ is bounded, there exists an $m\in C([t_0,T];L^1(\Omega))$ which solves (\ref{xxxxx}) in the sense of Definition 4.1, such that $m^{\epsilon}\rightarrow m$ in $C([t_0,T]; L^1(\Omega))$ as $\epsilon\rightarrow 0$. Imitating the argument in the proof of Proposition 3.3 of [2] (which is based on Lemma \ref{3.2}), we get that there exists a $C$, independent of $\epsilon$ and $L$, such that 
$$\bm{d_1}(m^{\epsilon}(t),m^{\epsilon}(s))\leq C|t-s|^{1/2},\text{ for any }t,s\in [t_0,T].$$
Now, let $\psi$ be a $1$-Lipschitz function defined in $\Omega$. We have
$$\int_{\Omega_{\epsilon}}\psi (m^{\epsilon}(t)-m^{\epsilon}(s))\leq \bm{d_1}(m^{\epsilon}(t),m^{\epsilon}(s))\leq C|t-s|^{1/2},$$
so passing to the limit as $\epsilon\rightarrow 0$
$$\int_{\Omega}\psi (m(t)-m(s))\leq C|t-s|^{1/2}.$$
Taking the supremum over $\psi$ yields
$$\bm{d_1}(m(t),m(s))\leq C|t-s|^{1/2},$$
hence, choosing $L=C$, we deduce that $m\in X$ and the map $T$ is well defined. Moreover, using a similar argument as in Proposition \ref{5.6.[2]} (relation (\ref{5.6.2})), we get that there exists a $\tilde{C}>0$, independent of $\epsilon$, such that 
$$||m^{\epsilon}||_{L^p}\leq \tilde{C},$$
so, by weak converegence, 
\begin{equation}\label{yyyyy}
||m||_{L^p(Q_T)}\leq \tilde{C}\;.
\end{equation}
\vspace{2mm}

\noindent
\textit{Step 2:} In order to apply Schauder's theorem and show that $T$ has a fixed point, it suffices to show that $T$ is continuous. For that, let $(\mu_n)_{n\in\mathbb{N}}$ be a sequence in $X$ such that $\mu_n\rightarrow \mu$. We will show that $m_n=T(\mu_n)\rightarrow T(\mu)=m$. We argue along the lines of Theorem \ref{MFG sol}. Fix $\psi\in W^{1,\infty}(\Omega)$ with $||\psi||_{W^{1,\infty}}\leq 1$ and $s\in (t_0,T]$. By Proposition 3.4 of [1], there exists $\phi^s_{\epsilon}$ that solves
$$\begin{cases}
-\partial_t \phi^s_{\epsilon}-\text{tr}(a(x)D^2\phi^s_{\epsilon})+H_p(x,Du_n)\cdot D\phi^s_{\epsilon}=0,\quad\text{ in } (t_0,s)\times \Omega_{\epsilon},\\
\phi^s_{\epsilon}(s)=\psi,\text{ in }\Omega,\\
a(x)D\phi^s_{\epsilon}\cdot\nu(x)|_{\partial\Omega_{\epsilon}}=0,
\end{cases}$$
such that $\phi^s_{\epsilon}\rightarrow \phi^s_n$ in $C([t_0,s]; L^p(\Omega))$ as $\epsilon\rightarrow 0$, where $\phi^s_n$ is the weak solution to the equation
$$\begin{cases}
    -\partial_t\phi_n^s-\text{tr}(a(x)D^2\phi_n^s )+H_p(x,f(t,x,\mu_n))\cdot D\phi_n^s=0,\text{ in }(t_0,s)\times\Omega,\\
    \phi_n (s)=\psi,\text{ in }\Omega,\\
    \end{cases}$$
under $(2.1)$. However, by Lemma \ref{3.2}, we have
$$|\phi^s_{\epsilon}(t,x)-\phi^s_{\epsilon}(w,y)|\leq C||\psi||_{W^{1,\infty}}(|t-w|^{1/2}+|x-y|),$$
where $C$ is independent of $\epsilon$, due to Proposition \ref{domains}. Thus, by Arzela-Ascoli,  $\phi^s_{\epsilon}\rightarrow \phi^s_n$ uniformly as $\epsilon\rightarrow 0$ and
\begin{equation}
    \sup_{t\in[t_0,s]}||\phi^s_n(t,\cdot)||_{W^{1,\infty}}\leq C||\psi||_{W^{1,\infty}}.
\end{equation}
Set $\phi^s_{n,k}:=\phi^s_n-\phi^s_k$. This function satisfies
$$\begin{cases}
-\partial_t \phi^s_{n,k}-\text{tr}(a(x)D^2\phi^s_{n,k})+H_p(x,f(t,x,\mu_n))\cdot D\phi^s_{n,k}=\\
\hspace{5cm}(H_p(x,f(t,x,\mu_k))-H_p(x,f(t,x,\mu_n)))\cdot D\phi_k^s,\text{ in }(t_0,T)\times\Omega,\\
\phi^s_{n,k}(s)=0,\text{ in }\Omega.
\end{cases}$$
By Proposition 3.4 of [1], once again, we know that $\phi_{n,k}^s$ can be obtained as a limit of $\phi_{n,k}^{s,\epsilon}$ solving the same equation in $\Omega_{\epsilon}$ with Neumann conditions. Therefore, by Lemma \ref{3.1}, we get that there exists a constant $C>0$, such that
\begin{equation}||\phi^s_{n,k}||_{\frac{1+\alpha}{2},1+\alpha}\leq C||(H_p(x,f(t,x,\mu_k))-H_p(x,f(t,x,\mu_n)))\cdot D\phi_k||_{\infty}\leq C\sup_{t\in [t_0,T]}\bm{d_1}(\mu_n(t),\mu_k(t)),\end{equation}
where in the last step we used (6.7) and the Lipschitz continuity of $H_p$ and $f$. Subtracting the weak formulations for $T(\mu_n)=m_n$ and $T(\mu_k)=m_k$ we thus get
$$\int_{\Omega}\psi(x)(m_n(s)-m_k(s))dx=\int_{\Omega}(\phi_n^s(t_0,x)-\phi_k^s(t_0,x))m_0dx\stackrel{(6.8)}{\leq} C\sup_{t\in [t_0,T]}\bm{d_1}(\mu_n(t),\mu_k(t))\;.$$
Therefore, taking supremum over $\psi$ and $s$,
$$\sup_{s\in [t_0,T]}\bm{d_1}(m_n(s),m_k(s))\leq C\sup_{t\in [t_0,T]}\bm{d_1}(\mu_n(t),\mu_k(t)).$$
We deduce that $(m_n)_{n\in\mathbb{N}}$ is Cauchy and hence it converges. In addition, due to (6.6), $(m_n)_{n\in\mathbb{N}}$ converges weakly in $L^p$, up to a subsequence, to the same limit. We conclude that $(m_n)_{n\in\mathbb{N}}$ converges to a solution of (\ref{xxxxx}), in the sense of Definition 4.1, and hence, by uniqueness, it converges to $m=T(\mu)$.
\vspace{2mm}

\noindent
\textit{Step 3:} We finally prove that $m$, the fixed point of $T$, is in $C^{1+\frac{\alpha}{2},2+\alpha}$. Indeed, since $m\in X$, we know that $f(\cdot,x,m(\cdot))$ is at least of class $C^{\frac{\alpha}{2}}$. It follows as in Theorem 4.4 (Step 2), that $m^{\epsilon}$, solving (6.5) for $\mu=m$, is of class  $C^{1+\frac{\alpha}{2},2+\alpha}(Q^{\epsilon}_T)$ with a uniform in $\epsilon$ $C^{1+\frac{\alpha}{2},2+\alpha}$-bound. Therefore, by Arzela-Ascoli, $m^{\epsilon}\rightarrow m$ locally uniformly in $C^{1,2}$ and $m\in C^{1+\frac{\alpha}{2},2+\alpha}$. The proof is complete.
\end{proof}

\vspace{5mm}

\noindent
We are now ready to prove the uniqueness (\textbf{Theorem} 1.2):

\vspace{3mm}

\begin{proof}

Suppose that $m_0$ is smooth with compact support. Let $V$ be another classical solution of \textbf{ME} and let $\tilde{m}\in C^{1+\frac{\alpha}{2},2+\alpha}$ be a solution to the problem
$$\begin{cases}
\partial_tm-\text{div}(a(x)Dm)-\text{div}\left(m(H_p(x,D_xV(t,x,m))+\tilde{b}(x))\right)=0\quad ,\text{ in }[t_0,T]\times\Omega,\\
m(t_0)=m_0,
\end{cases}$$
given by Lemma \ref{General FP}.

\noindent
Set $\tilde{u}(t,x)=V(t,x,\tilde{m}(t))$. Using the chain rule and Lemma \ref{int by parts}, we can do the following calculation
\begin{align*}
\partial_t\tilde{u}(t,x)&=\partial_t V(t,x,\tilde{m}(t))+\int_{\Omega} \frac{δV}{δm}(t,x,\tilde{m}(t),\cdot )\partial_t \tilde{m}(t)\;dy\\
&=\partial_t V(t,x,\tilde{m}(t))+\int_{\Omega}\text{tr}(a(y)D_yD_mV(t,x,\tilde{m}(t),y))\tilde{m}(t)\;dy
\end{align*}
$$\quad\quad\quad\quad\quad\quad\quad\quad\quad\quad-\int_{\Omega}D_mV(t,x,\tilde{m}(t),y)\cdot H_p(y,D_xV(t,y,\tilde{m}))\tilde{m}(t)\; dy\;.$$
\vspace{1mm}

\noindent
Recalling that $V$ satisfies the \textbf{ME}, we find
$$\partial_t \tilde{u}(t,x)=-\text{tr}(a(x)D^2_x V(t,x,\tilde{m}(t)))+H(x,D_xV(t,x,\tilde{m}(t))-F(x,\tilde{m}(t))$$
$$=-\text{tr}(a(x)D^2\tilde{u}(t,x))+H(x,D\tilde{u}(t,x))-F(x,\tilde{m}(t)).$$
We conclude that $(\tilde{u},\tilde{m})$ solves, in the sense of Proposition \ref{MFGL1}, the \textbf{(MFG)} system with initial conditions $(t_0,m_0)$. By uniqueness (Lemma \ref{stability estimates}), we deduce that $u=\tilde{u}$ and $m=\tilde{m}$, hence $U(t_0,\cdot,m_0)=V(t_0,\cdot,m_0)$ for every $m_0$ smooth with compact support. It follows, by a density argument, that $U=V$. The proof is complete.
\end{proof}

\vspace{7mm}

\section{The Convergence Problem}

\noindent
In this section, we establish a connection between the \textbf{ME} and the asymptotic behaviour of the $N-$player differential game as $N\rightarrow +\infty$.

\vspace{4mm}

\subsection{The Nash System}
Before we proceed with the main results we need to establish the well-posedness of the Nash system (\ref{NS}) over $\Omega^N$, where we assume (\ref{PR}) near the boundary. Namely,

\vspace{3mm}

\begin{thm}
Assume Hypotheses 2.3. Then, the following system, $1\leq i\leq N$, has a unique classical solution.
\begin{equation}\label{NS1}
\begin{cases}
-\partial_t v_i^{N,\epsilon}-\sum_j \text{tr}(a(x_j)D^2_{x_jx_j}v_i^{N,\epsilon}(t,\bm{x}))+H(x_i,D_{x_i}v_i^{N,\epsilon}(t,\bm{x}))\\
\quad\quad\quad\quad+\sum_{j\neq i}H_p(x_j,D_{x_j}v_j^N(t,\bm{x}))\cdot D_{x_j}v_i^{N,\epsilon}(t,\bm{x})=F(x_i,m_{\bm{x}}^{N,i}),\text{ in }[0,T]\times \Omega^N,\\
v_i^{N,\epsilon}(T,\bm{x})=G(x_i,m_{\bm{x}}^{N,i})\quad\text{ in }\Omega^N.
\end{cases}
\end{equation}
\end{thm}
\vspace{1mm}

\begin{proof}
Indeed, consider the Nash system for the $N-$player game with Neumann conditions over $\Omega_{\epsilon}$. It takes the form
$$\begin{cases}
-\partial_t v_i^{N,\epsilon}-\sum_j \text{tr}(a(x_j)D^2_{x_jx_j}v_i^{N,\epsilon}(t,\bm{x}))+H(x_i,D_{x_i}v_i^{N,\epsilon}(t,\bm{x}))\\
\quad+\sum_{j\neq i}H_p(x_j,D_{x_j}v_j^N(t,\bm{x}))\cdot D_{x_j}v_i^{N,\epsilon}(t,\bm{x})=F(x_i,m_{\bm{x}}^{N,i}),\text{ in }(0,T)\times\Omega_{\epsilon}^N,\\
v_i^{N,\epsilon}(T,\bm{x})=G(x_i,m_{\bm{x}}^{N,i})-g_{\epsilon,i}(\bm{x}),\\
a(x_j)D_{x_j}v_i^{N,\epsilon}(t,\bm{x})\cdot \nu (x_j)|_{x_j\in\partial\Omega_{\epsilon}}=0,\text{ for all }j\in \{1,2,...,N\}
\end{cases},\; 1\leq i\leq N,$$
where $g_{\epsilon,i}(\bm{x})=\mathcal{N}_{\epsilon}(G(\cdot,m^{N,i}_{\cdot})(\bm{x})$ ensures that the compatibility conditions hold.
By results from [4], this system has a unique solution in $C^{1,2}(\overline{Q_T^{\epsilon}})$ and, in particular, this solution satisfies a uniform $C^{1+\frac{\alpha}{2},2+\alpha}$ estimate (Theorem V 7.4 in [4] and remark after Theorem VII 7.1 in [4])
$$||v_i^{N,\epsilon}||_{C^{1+\frac{\alpha}{2},2+\alpha}}\leq C,\text{ in }\Omega_{\epsilon},$$
where $C$ is independent of $\epsilon$, but depends on $N$. Therefore, by Arzela-Ascoli, there exists a subsequence such that
$$v_i^{N,\epsilon}\rightarrow v_i^N\text{ locally uniformly in }C^{1,2}\text{ as }\epsilon\rightarrow 0.$$
By passing to the limit, it is easy to see that $v_i^{N}$, $i=1,2,...,N$, is a classical solution to (\ref{NS1}). Uniqueness follows using the same technique as in Theorem V 6.1 in [4].
\end{proof}

\vspace{2mm}

\subsection{Finite dimensional projections and their properties}
For $N\geq 2$ and $1\leq i\leq N$, we define 
$$u^N_i(t,\bm{x}):=U(t,x_i,m_{\bm{x}}^{N,i}),$$
where $m_{\bm{x}}^{N,i}$ is the empirical distribution of the players $j\neq i$. From the regularity properties of $U$, we know that $u^N_i(t,x)\in C^{1,2+\alpha}$. The following proposition illustrates the connection between the derivatives of $u^N_i(t,x)$ and the derivatives of $U$.
\vspace{2mm}
\begin{prop}\label{4}
For all $j\neq i$, the following formulas hold:
\begin{align*}
& D_{x_j}u^N_i(t,\bm{x})=\frac{1}{N-1}D_mU(t,x_i,m_{\bm{x}}^{N,i},x_j),\\
& D^2_{x_i,x_j}u^N_i(t,\bm{x})=\frac{1}{N-1}D_xD_mU(t,x_i,m_{\bm{x}}^{N,i},x_j),\\
& \left| D^2_{x_j,x_j}u^N_i(t,\bm{x})-\frac{1}{N-1}D_yD_mU(t,x_i,m_{\bm{x}}^{N,i},x_j)\right|\leq \frac{C}{N^2}.
\end{align*}
\end{prop}

\begin{proof}
Having the relations in Theorem \ref{Lip dU} in our disposal, the proof is a trivial adaptation of the proof of Proposition 4.1 of [16].
\end{proof}

\vspace{4mm}

\noindent
Now, we are ready to prove that $u^N_i(t,x)$, $i=1,2,...,N$ is ``almost'' a solution to the Nash system (\ref{NS1}), through its connection to the \textbf{ME}.

\begin{thm}
Assume Hypotheses 2.3 hold. Then, $u^N_i(t,x)$ satisfies the following equation :
$$\begin{cases}
-\partial_t u_i^{N,\epsilon}-\sum_j \text{tr}(a(x_j)D^2_{x_jx_j}u_i^{N,\epsilon}(t,\bm{x}))+H(x_i,D_{x_i}u_i^{N,\epsilon}(t,\bm{x}))\\
\quad\quad\quad\quad+\sum_{j\neq i}H_p(x_j,D_{x_j}u_j^N(t,\bm{x}))\cdot D_{x_j}u_i^{N,\epsilon}(t,\bm{x})=F(x_i,m_{\bm{x}}^{N,i})+r_i^N(t,x),\\
\quad\quad\quad\quad\quad\quad\quad\quad\quad\quad\quad\quad\quad\quad\quad\quad\quad\quad\quad\quad\quad\quad\quad\quad\quad\quad\quad\quad\text{in }[0,T]\times \Omega^N,\\
u_i^{N,\epsilon}(T,\bm{x})=G(x_i,m_{\bm{x}}^{N,i})\quad\text{ in }\Omega^N,
\end{cases}$$
where $r_i^N\in L^{\infty}$ with $\left|\left|r_i^N\right|\right|_{\infty}\leq \frac{C}{N}$.
\end{thm}

\begin{proof}
Since $u^N_i(t,x):=U(t,x_i,m_{\bm{x}}^{N,i})$, we evaluate the ME -satisfied by $U$- at $(t,x_i,m_x^{N,i})$ and we find
$$-\partial_tu^N_i(t,x)-\text{tr}(a(x_i)D^2_{x_i,x_i}u^N_i)+H(x_i,D_{x_i}u_i^N)-\int_{\Omega}\text{tr}(a(y)D_yD_mU(t,x_i,m_x^{N,i},y))dm_x^{N,i}(y)$$
$$+\frac{1}{N-1}\sum_{j\neq i} H_p(x_j,D_xU(t,x_j,m_x^{N,i}))\cdot D_m U(t,x_i,m_x^{N,i},x_j)=F(x_i,m_x^{N,i}).$$
Using the derivative formulas from Proposition \ref{4} on the fifth and the third term from the left-hand side, the result follows.
\end{proof}

\vspace{7mm}

\subsection{Convergence Results}

\noindent
Let $v_i^N$, $1\leq i\leq N$, be the solution to (\ref{NS1}) for the $N$-player differential game. By definition and the symmetry of (\ref{NS1}), the functions $u_i^n$ and $v_i^n$ are symmetrical, that is, there exist two functions $U^N$ and $V^N: \Omega\times \Omega^{N-1}\rightarrow \mathbb{R}$, which are invariant under all permutations of the last $N-1$ coordinates and 
$$u_i^N(t,\bm{x})=U^N(x_i,(x_1,...x_{i-1},x_{i+1},...,x_N)),$$
$$v_i^N(t,\bm{x})=V^N(x_i,(x_1,...x_{i-1},x_{i+1},...,x_N)).$$
Let us fix $t_0\in [0,T)$, $m_0\in\mathcal{P}(\Omega)$ and $\bm{Z}=(Z^i)_i$ an i.i.d family of $N$ random variables of law $m_0$. The convergence results rely on the comparison of $v_i^N$ and $u_i^N$ along the trajectories defined by the following systems of SDEs:
$$\begin{cases}
dY^i_t=-H_p(Y^i_t,D_{x_i}v_i^N(t,\bm{Y}_t))dt+\sqrt{2}\sigma(Y^i_t)dB^i_t,\\
Y^i_{t_0}=Z_i
\end{cases}$$
and
$$\begin{cases}
dX^i_t=-H_p(X^i_t,D_{x_i}u_i^N(t,\bm{X}_t))dt+\sqrt{2}\sigma(X^i_t)dB^i_t,\\
Y^i_{t_0}=Z_i.
\end{cases}$$
The following theorem is true.
\vspace{4mm}

\begin{thm}
Suppose Hypotheses 2.3 are satisfied. Then, for any $1\leq i\leq N$, we have
$$\mathbb{E}\left[ \int_{t_0}^T|D_{x_i}v_i^N(t,\bm{Y}_t)-D_{x_i}u_i^N(t,\bm{Y}_t)|^2dt\right]\leq \frac{C}{N^2}.$$
Moreover, $\mathbb{P}$-a.s,
$$|u_i^N(t_0,\bm{Z})-v_i^N(t_0,\bm{Z})|\leq \frac{C}{N}.$$
\end{thm}

\begin{proof}
The proof goes along the lines of the proof of Theorem 6.2.1 in [5], so we only provide a sketch. We define the processes 
$$U^{N,i}_t:=u_i^N(t,\bm{Y}_t)\text{ and }V^{N,i}_t:=v_i^N(t,\bm{Y}_t)$$
with the notation $U^{N,i,j}_t=D_{x_j}U^{N,i}_t$ and $V^{N,i,j}=D_{x_j}V^{N,i}_t$. We use Itô's formula to the process $(U_t^{N,i}-V_t^{N,i})^2$. Integrating and using Grönwall's lemma gives
$$\sup_{t\in[t_0,T]}\left[ \sum_{i=1}^N\mathbb{E}^{\bm{Z}}[|U_t^{N,i}-V_t^{N,i}|^2]\right]+\sum_{i=1}^N\mathbb{E}^{\bm{Z}}\left[\int_0^T|DU_t^{N,i,i}-DV_t^{N,i,i}|^2dt\right]\leq \frac{C}{N}$$
and
\begin{equation} \label{3}
    \sup_{t\in [t_0,T]}\mathbb{E}^Z[|U_t^{N,i}-V_t^{N,i}|^2]+\mathbb{E}^Z\left[ \int_t^T|DU_s^{N,i,i}-DV_s^{N,i,i}|^2ds\right]\leq \frac{C}{N^2}.
\end{equation}
Evaluating the term in the sup at $t=t_0$, we prove the second inequality. Taking the integral at $t_0$, proves the first inequality.
\end{proof}
\vspace{3mm}

\noindent
Using these inequalities, the proof of Theorem 1.3 is a trivial adaptation of the proof of Theorem 2.4.8 in [5].

\vspace{4mm}

\noindent
Finally, we prove Theorem 1.4, which concerns the convergence of the trajectories. 
\vspace{2mm}

\begin{proof}
Note that
\begin{multline*}
    (X^i_t-Y^i_t)^2\\
    \leq \left( \int_0^t(H_p(s,X^i_s,D_{x_i}u^{N,i}(s,\bm{X}_s))-H_p(s,Y^i_s,D_{x_i}v^{N,i}(s,\bm{Y}_s)))ds\right)^2\\
    +\left( \int_0^t\sqrt{2}(\sigma(X^i_s)-\sigma(Y^i_s))dB_s^i\right)^2.
\end{multline*}
Since $H_p$ is Lipschitz and $U$ is regular, using Proposition \ref{4} we have
\begin{multline*}
    (X^i_t-Y^i_t)^2\leq C\left[ \int_0^t\left(|X^i_s-Y^i_s|+\frac{1}{N}\sum_{j\neq i}|X^j_s-Y^j_s|\right)ds\right]^2 \\
    +C\left[ \int_0^t(H_p(s,Y^i_s,D_{x_i}u^{N,i}(s,\bm{Y}_s))-H_p(s,Y^i_s,D_{x_i}v^{N,i}(s,\bm{Y}_s)))ds \right] ^2\\
    +\left( \int_0^t\sqrt{2}(\sigma(X^i_s)-\sigma(Y^i_s))dB_s^i\right)^2.
\end{multline*}
We now take the conditional expectation over $\bm{Z}$.
\begin{multline}
    \mathbb{E}^Z[(X^i_t-Y^i_t)^2]\\
    \leq C\int_0^t\left(\mathbb{E}^Z[|X^i_s-Y^i_s|^2]+\frac{1}{N}\sum_{j\neq i}\mathbb{E}^Z[|X^j_s-Y^j_s|^2]\right)ds\\
    +\mathbb{E}^Z\left[ \int_0^T|DU^{N,i,i}_s-DV_s^{N,i,i}|^2ds\right]+\mathbb{E}^Z\left[ \left( \int_0^t\sqrt{2}(\sigma(X^i_s)-\sigma(Y^i_s))dB_s^i\right)^2  \right].
\end{multline}
From the Itô isometry and the Lipschitz assumption on $\sigma$, we get
\begin{align*}
     \mathbb{E}^Z\left[ \left( \int_0^t\sqrt{2}(\sigma(X^i_s)-\sigma(Y^i_s))dB_s^i \right) ^2  \right]&=\mathbb{E}^Z\left[\int_0^t 2(\sigma(X^i_s)-\sigma(Y^i_s))^2ds]\right]\\
    &\leq C\int_0^t\mathbb{E}^Z[(X^i_s-Y^i_s)^2]ds.
\end{align*}
We now plug this in $(7.3)$, sum over $i=1,2,...,N$ and use $(7.2)$. An application of Grönwall's lemma implies
$$\sup_{t\in [0,T]}\sum_{i=1}^N\mathbb{E}^Z[(X^i_t-Y^i_t)^2]\leq \frac{C}{N}.$$
Inserting this in $(7.3)$ and using $(7.2)$ once again, Grönwall's lemma implies the result.
\end{proof}

\addtocontents{toc}{\protect\setcounter{tocdepth}{0}}

\section*{Acknowledgements}  
\noindent
The author wishes to thank Professor P.E. Souganidis for valuable discussions, comments, and suggestions. During the course of this work, the author was partially supported by P.E. Souganidis’ NSF grant DMS-1900599, ONR grant N000141712095 and AFOSR grant FA9550-18-1-0494. The author would like to thank the Institute for Mathematical and Statistical Innovation for its hospitality.

\addtocontents{toc}{\protect\setcounter{tocdepth}{5}}


\begin{thebibliography}{3}


\bibitem[1]{} Porretta, Alessio and Ricciardi, Michele (2019). Mean field games under invariance conditions for the state space.
arxiv.1903.06491
 
 
 \bibitem[2]{} Ricciardi, M. (2021). The Master Equation in a Bounded Domain with
Neumann Conditions. Communications in Partial Differential Equations,
DOI: 10.1080/03605302.2021.2008965.
 
 \bibitem[3]{} Allen Bensoussan, Jens Frehse, Philipp Yam 
 Mean Field Games and Mean Field Type Control Theory. (Springer, 2013)
 
 
 \bibitem[4]{} Ladyzenskaja O.A., Solonnikov, V.A, and Ural’ceva, N.N Linear and quasilinear equations of
parabolic type. Translations of Mathematical Monograms, Vol 23 American Mathematical
Society, Providence, R.I. 1967

 
 \bibitem[5]{} P. Cardaliaguet, F. Delarue, J.-M Lasry, P.-L Lions
 The master equation and the convergence problem in mean field games. (2019)


 \bibitem[6]{} Lunardi, A. (2012). Analytic Semigroups and Optimal Regularity in Parabolic Problems. Modern
Birkhäuser Classics
 
 
 \bibitem[7]{} D. Gilbarg, N.S. Trudinger, Elliptic Partial Differential Equations of Second Order (Springer,
Berlin, 2015)

 \bibitem[8]{}  Lions, P.-L. Cours au Collège de France. www.college-de-france.fr.

\bibitem[9]{}  Di Persio, Luca , Garbelli, Matteo , Ricciardi, Michele (2022).
The Master Equation in a Bounded Domain with Absorption

\bibitem[10]{}Lasry, J.-M., Lions, P.-L.: Jeux à champ moyen. I. Le cas stationnaire. C. R. Math. Acad. Sci. Paris 343, 619-625 (2006)

\bibitem[11]{} Lasry, J.-M., Lions, P.-L.: Jeux à champ moyen. II. Horizon fini et contrôle optimal. C. R. Math. Acad. Sci.    Paris 343, 679-684 (2006)

\bibitem[12]{} Lasry, J.-M., Lions, P.-L. Mean feld games. Jpn. J. Math. 2 (2007), no. 1, 229-260

\bibitem[13]{} Huang, Caines and Malhame R.P.(2006) Large population stochastic dynamic games: closed loop McKean-Vlasov systems and the Nash certainty equivalence principle. Communication in information and systems. Vol. 6, No. 3, pp. 221-252

\bibitem[14]{}Delfour, M.C., Zolesio, J.-P. (1994). Shape analysis via oriented distance function. J. Funct. Anal. 123, 129-201

\bibitem[15]{} Kirszbraun, M. D. (1934). "Über die zusammenziehende und Lipschitzsche Transformationen". Fundamenta Mathematicae. 22: 77–108. doi:10.4064/fm-22-1-77-108

\bibitem[16]{} Ricciardi, M. (2022). The Convergence Problem in Mean Field Games with Neumann Conditions. arXiv:2203.07882.

\bibitem[17]{} Sardarli, Mariya (2021). The ergodic Mean Field Game system for a type of state constraint condition. arxiv.2101.10564



\end{thebibliography}
\end{document}